\documentclass{amsproc}
\usepackage{amssymb,tikz}
\usepackage{verbatim}
\usepackage{graphicx}
\usepackage[cmtip,all]{xy}
\usepackage{amsmath}
\usepackage{amsfonts,amssymb,amsthm}
\usepackage{times}
\usepackage{amscd}
\usepackage{epsfig}
\usepackage{float}

\newtheorem{theo}{Theorem}[section]
\newtheorem{thm}[theo]{Theorem}
\newtheorem{lem}[theo]{Lemma}
\newtheorem{cor}[theo]{Corollary}

\newtheorem*{thm*}{Theorem}

\theoremstyle{definition}
\newtheorem{dfn}[theo]{Definition}
\newtheorem{exa}[theo]{Example}

\theoremstyle{remark}

\numberwithin{equation}{section}

\newcommand{\cri}{\mathrm{Cr}}
\newcommand{\di}{\mathrm{Di}}
\newcommand{\ta}{\theta}
\newcommand{\hc}{\mbox{$\mathbb{\widehat{C}}$}}
\newcommand{\iU}{U^\infty}
\newcommand{\ql}{quadrilateral{}}

\newcommand{\C}{\mathbb{C}}

\newcommand{\disk}{\mathbb{D}}

\newcommand{\cdisk}{\ol{\mathbb{D}}}
\newcommand{\bbd}{\mathbb{D}}

\newcommand{\ol}{\overline}

\newcommand{\sm}{\setminus}

\newcommand{\m}{\ol{m}}
\newcommand{\n}{\ol{n}}
\newcommand{\hell}{\hat{\ell}}

\newcommand{\qml}{\mathrm{QML}}
\newcommand{\bd}{\mathrm{Bd}}

\newcommand{\lam}{\mathcal{L}}
\newcommand{\hlam}{\widehat{\mathcal{L}}}

\newcommand{\ch}{\mathrm{CH}}

\newcommand{\si}{\sigma}
\newcommand{\0}{\emptyset}
\newcommand{\uc}{\mathbb{S}}

\newcommand{\g}{\mathfrak{g}}

\newcommand{\e}{\varepsilon}
\newcommand{\M}{\mathcal{M}}

\newcommand{\laq}{\mathbb{L}^q}
\newcommand{\olq}{\ol{\mathbb{L}^q_2}}

\newcommand{\Mcc}{\M^{comb}}

\begin{document}
\date{February 25, 2015, and, in revised form, June 3, 2015}

\title[The Mandelbrot set and the space of geolaminations]
{The combinatorial Mandelbrot set as the quotient of the space of geolaminations}

\author{Alexander~Blokh}

\address[Alexander~Blokh]
{Department of Mathematics\\ University of Alabama at Birmingham\\
Birmingham, AL 35294}

\email{ablokh@math.uab.edu}

\thanks{The first named author was partially
supported by NSF grant DMS--1201450 and MPI f\"ur Mathematik, Bonn,
activity ``Dynamics and Numbers''}

\author{Lex Oversteegen}

\address[Lex Oversteegen]
{Department of Mathematics\\ University of Alabama at Birmingham\\
Birmingham, AL 35294}

\email{overstee@uab.edu}


\author{Ross~Ptacek}

\address[R.~Ptacek]
{Faculty of Mathematics\\
Laboratory of Algebraic Geometry and its Applications\\
Higher School of Economics\\
Vavilova St. 7, 112312 Moscow, Russia}

\curraddr{Department of Mathematics, 1400 Stadium Rd, University of
Florida, Gainesville, FL 32611}

\email{rptacek@ufl.edu}

\author{Vladlen~Timorin}

\address[V.~Timorin]{Faculty of Mathematics\\
Laboratory of Algebraic Geometry and its Applications\\
Higher School of Economics\\
Vavilova St. 7, 112312 Moscow, Russia }

\thanks{The fourth named author was partially supported by
the Dynasty foundation grant, the Simons-IUM fellowship, RFBR grants
11-01-00654-a, 12-01-33020.
The article was prepared within the framework of a subsidy granted to the HSE 
by the Government of the Russian Federation for the implementation of the Global Competitiveness Program.}



\begin{abstract} We interpret the combinatorial Mandelbrot set in terms
of \emph{quadratic laminations} (equivalence relations $\sim$
on the unit circle invariant under $\si_2$). To each lamination we
associate a particular \emph{geolamination} (the collection $\lam_\sim$
of points of the circle and edges of convex hulls of $\sim$-equivalence
classes) so that the closure of the set of all of them is a compact metric space with
the Hausdorff metric. Two such geolaminations are said to be
\emph{minor equivalent} if their \emph{minors} (images of their longest
chords) intersect. We show that the corresponding quotient space of
this topological space is homeomorphic to the boundary of the
combinatorial Mandelbrot set. To each equivalence class  of these
geolaminations we associate a unique lamination and its topological
polynomial so that this interpretation can be viewed as a way to endow
the space of all quadratic topological polynomials with a suitable
topology.
\end{abstract}

\maketitle

\section*{Introduction}\label{s:intro}

Studying the structure of polynomial families is one of the central
problems of complex dynamics. The first non-trivial case here is
that of quadratic polynomial family $P_c(z)=z^2+c$. The
\emph{Mandelbrot set $\M_2$} is defined as the set of the parameters
$c$ such that the trajectory of the critical point $0$ of $P_c$ does
not escape to infinity under iterations of $P_c$. Equivalently, this
is the set of all parameters $c$ such that the Julia set $J(P_c)$ of
$P_c$ is connected.

Thurston \cite{thu85} constructed a combinatorial model for $\M_2$,
which can be interpreted as follows. \emph{Laminations} are closed
equivalence relations $\sim$ on the unit circle $\uc$ in the complex
plane $\C$ such that all classes are finite and the convex hulls of all
classes are pairwise disjoint. A lamination is said to be
\emph{($\si_d$-) invariant} if it is preserved under the map
$\si_d(z)=z^d:\uc\to\uc$ (precise definitions are given in the next
section; if no ambiguity is possible, we will simply talk about \emph{invariant} laminations).
The map $\si_d$ induces a \emph{topological polynomial}
$f_\sim:\uc/\sim\to \uc/\sim$ from the \emph{topological Julia set}
$J_\sim=\uc/\sim$ to itself. If $J(P_c)$ is locally connected, then
$P_c|_{J(P_c)}$ is conjugate to $f_\sim$ for a specific lamination
$\sim$. If $d=2$, then corresponding laminations, topological
polynomials and Julia sets are said to be \emph{quadratic}. The precise
definitions are given later in the paper (in particular, topological polynomials
and topological Julia sets are defined in Subsection~\ref{ss:lam}).

Even though $\M_2\subset \C$ has a natural topology, the proper
topology on the set of all quadratic topological polynomials is much
more elusive. Thurston constructed a suitable topology on this set by
associating to each (quadratic) lamination $\sim$ a geometric object
$\lam_\sim$, called a (quadratic) \emph{geolamination}, which consists
of all chords in the boundaries of the convex hulls (in the closed unit
disk) of all equivalence classes of $\sim$.

Quadratic geolaminations are invariant under the map $\si_2$. In
particular given a chord $\ell=\ol{ab}\in\lam$, with endpoints
$a,b\in\uc$, the chord $\ol{\si_2(a)\si_2(b)}$ is also a (possibly
degenerate) chord of $\lam$ and we write
$\si_2(\ell)=\ol{\si_2(a)\si_2(b)}$. Thurston parameterizes all such
geolaminations $\lam$ by the \emph{minors $m_\lam$ of $\lam$}; the
minor $m_\lam$ of a quadratic geolamination $\lam$ is the image of a
chord in $\lam$ of maximal length. Thurston shows that the
collection of minors of all $\si_2$-invariant geolaminations is
itself a geolamination, which we denote here by $\lam_\qml$. It
turns out that $\lam_\qml$ defines a lamination that Thurston called
$\qml$ (for \emph{quadratic minor lamination}). Leaves of
$\lam_\qml$ are exactly edges of convex hulls of classes of $\qml$.
The quotient $\uc/\qml$ provides the proper topology on the set of
quadratic topological polynomials, which serves as a model for the
boundary of $\M_2$. Observe that there exists a monotone map
$p:\bd(\M_2)\to \uc/\qml$ \cite{sch09} (cf \cite{thu85, kel00}).

\begin{figure}[!htb]
    \centering
    \begin{minipage}{0.5\textwidth}
        \centering
        \includegraphics[width=0.7\linewidth, height=0.2\textheight]{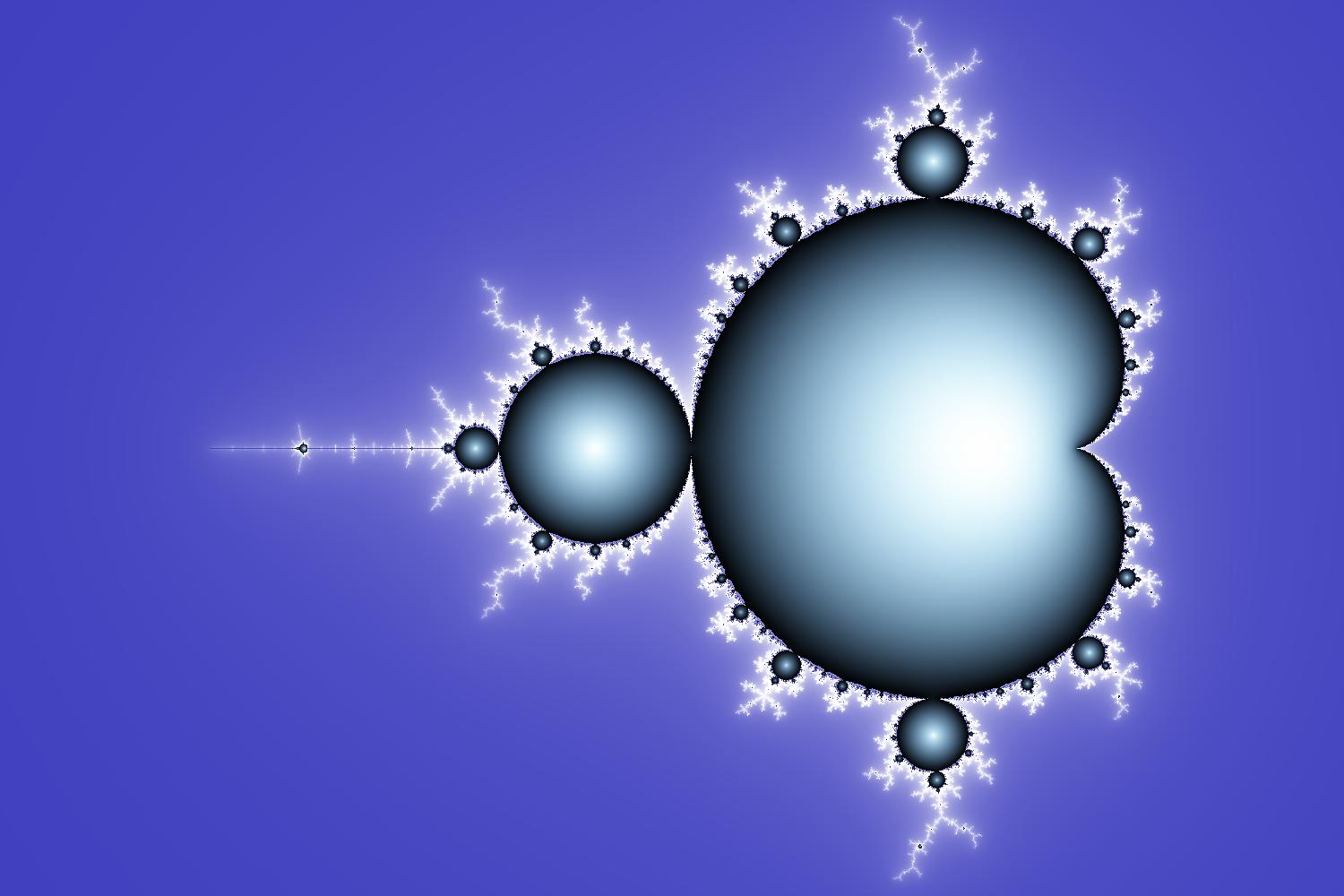}
        \caption{The Mandelbrot set}
        \label{fig:mset}
    \end{minipage}%
    \begin{minipage}{0.5\textwidth}
        \centering
        \includegraphics[width=0.7\linewidth, height=0.2\textheight]{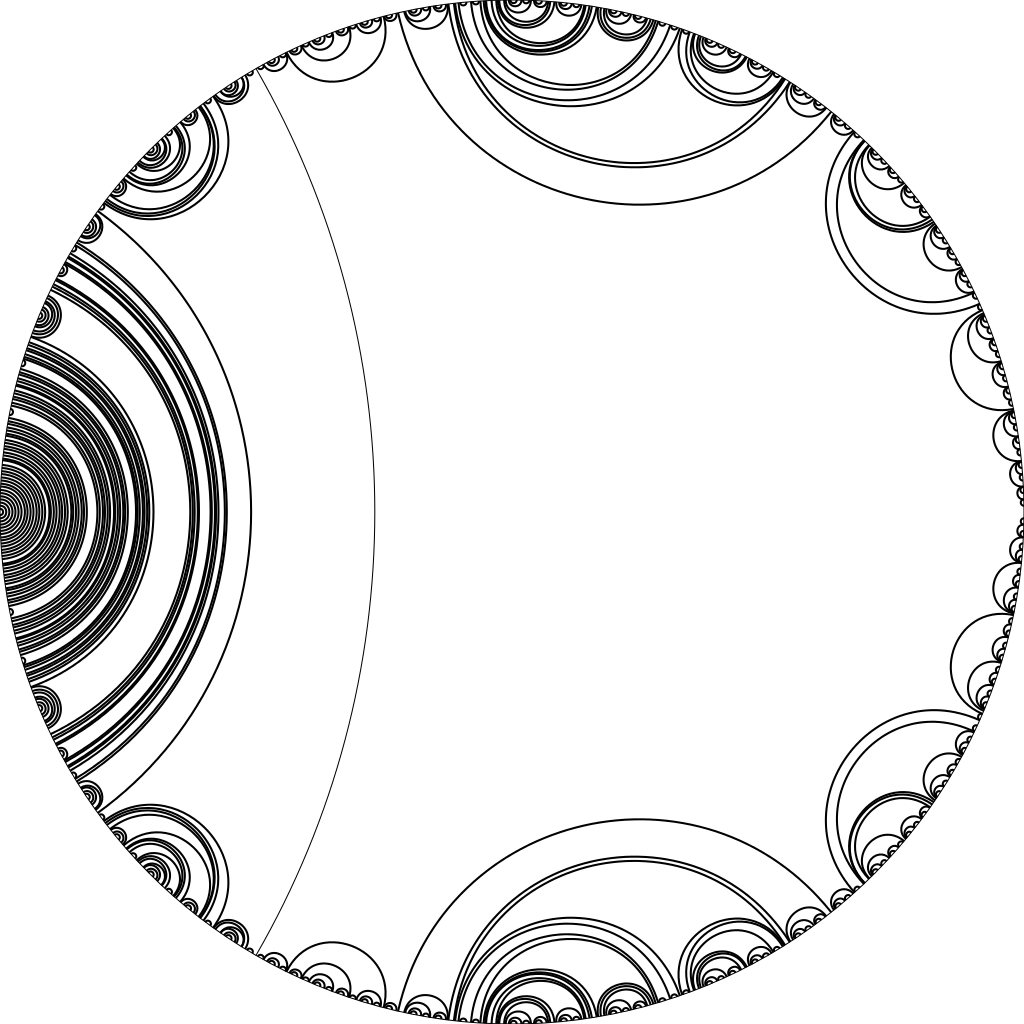}
        \caption{The (geo)lamination $\lam_\qml$}
        \label{fig:qml}
    \end{minipage}
\end{figure}

In the higher degree case no such parameterization of the set of
$\si_d$-invariant laminations is known. In particular, the proper
topology on the set of topological polynomials of degree $d$ is not
clear. Since the set of all $\si_d$-invariant geolaminations carries a
natural topology, induced by the Hausdorff distance between the unions
of their leaves, it is natural to impose this topology on the set of
geolaminations. It is well known that with this topology  the space of
geolaminations is a compact metric space. In this paper we consider a
suitable compact subspace of the space of all invariant quadratic
geolaminations with the Hausdorff metric. On this subspace we define a
closed equivalence relation with finite equivalence classes. We show
that each class corresponds to a unique $\si_2$-invariant lamination
and, hence, a unique quadratic topological polynomial. Finally we prove
that with the induced topology the corresponding quotient space is
homeomorphic to $\uc/\mathrm{QML}$.

The long term hope is to use the ideas from this paper to impose a
proper topology on the set of degree $d$ topological polynomials and
use it to obtain a combinatorial model for the boundary of the
connectedness locus $\M_d$ of the space of degree $d$ polynomials. The
analogy with polynomials is underlined by some of our rigidity results
for laminations, which parallel those for polynomials. For example, we
show in this paper that if the geolamination $\lam$ is a limit of geolaminations
$\lam_i$ corresponding to laminations $\sim_i$ and if $G$ is a
finite gap of $\lam$, then $G$ is also a gap $\lam_i$ for all large
$i$.

\noindent\textbf{Acknowledgements}. The paper was partially written as
the first named author was visiting Max-Planck-Institut f\"ur
Mathematik in Bonn during their activity ``Dynamics and Numbers''. He
would like to express his appreciation to the organizers and
Max-Planck-Institut f\"ur Mathematik for inspiring working conditions.
It is our pleasure to also thank the referee for thoughtful and useful
remarks.

\section{Preliminaries}\label{s:prelim}

A big portion of this section is devoted to (geo)laminations, a major
tool in studying both dynamics of individual complex polynomials and in
modeling certain families of complex polynomials (the
Mandelbrot set, which can be thought of as the family of all
polynomials $P_d=z^2+c$ with connected Julia set). Let $a$, $b\in \uc$.
By $[a, b]$, $(a, b)$, etc we mean the closed, open, etc
\emph{positively oriented} circle arcs from $a$ to $b$, and by $|I|$
the length of an arc $I$ in $\uc$ normalized so that the length of
$\uc$ is $1$.

\subsection{Laminations}\label{ss:lam}

Denote by $\hc$ the Riemann sphere.
For a compactum $X\subset\C$, let $\iU(X)$ be the unbounded component
of $\hc\sm X$ containing infinity. If $X$ is connected, there exists a
Riemann mapping $\Psi_X:\hc\sm\ol\bbd\to \iU(X)$; we always normalize
it so that $\Psi_X(\infty)=\infty$ and $\Psi'_X(z)$ tends to a positive
real limit as $z\to\infty$.

Consider a polynomial $P$ of degree $d\ge 2$ with Julia set $J_P$
and filled-in Julia set $K_P$. Extend $z^d=\si_d: \C\to \C$ to a map
$\ta_d$ on $\hc$. If $J_P$ is connected, then
$\Psi_{J_P}=\Psi:\hc\sm\ol\bbd\to \iU(K_P)$ is such that $\Psi\circ
\ta_d=P\circ \Psi$ on the complement of the closed unit disk
\cite{hubbdoua85, mil00}. If $J_P$ is locally connected, then $\Psi$
extends to a continuous function
$$
\ol{\Psi}: {\hc\sm\bbd}\to
\ol{\hc\setminus K_P},
$$
and $\ol{\Psi} \circ\,\ta_d=P\circ\ol{\Psi}$ on the complement of the
open unit disk; thus, we obtain a continuous surjection
$\ol\Psi\colon\bd(\bbd)\to J_P$ (the \emph{Carath\'eodory loop}); throughout the paper
by $\bd(X)$ we denote the boundary of a subset $X$ of a topological space.
Identify $\uc=\bd(\bbd)$ with $\mathbb{R}/\mathbb{Z}$. In this case set
$\psi=\ol{\Psi}|_{\uc}$.

Define an equivalence relation $\sim_P$ on $\uc$ by $x \sim_P y$ if
and only if $\psi(x)=\psi(y)$, and call it the ($\si_d$-invariant)
{\em lamination of $P$}; since $\Psi$ defined above conjugates
$\ta_d$ and $P$, the map $\psi$ semiconjugates $\si_d$ and
$P|_{J(P)}$, which implies that $\sim_P$ is invariant. Equivalence
classes of $\sim_P$ have pairwise disjoint convex hulls. The
\emph{topological Julia set} $\uc/\sim_P=J_{\sim_P}$ is homeomorphic
to $J_P$, and the \emph{topological polynomial}
$f_{\sim_P}:J_{\sim_P}\to J_{\sim_P}$, induced by $\si_d$, is
topologically conjugate to $P|_{J_P}$.


An equivalence relation $\sim$ on the unit circle, with similar
properties to those of $\sim_P$ above,  can be introduced abstractly
without any reference to the Julia set of a complex polynomial.

\begin{dfn}[Laminations]\label{d:lam}

An equivalence relation $\sim$ on the unit circle $\uc$ is called a
\emph{lamination} if it has the following properties:

\noindent (E1) the graph of $\sim$ is a closed subset in $\uc \times
\uc$;

\noindent (E2) convex hulls of distinct equivalence classes are
disjoint;


\noindent (E3) each equivalence class of $\sim$ is finite. 
\end{dfn}

For a closed set $A\subset \uc$ we denote its convex hull by
$\ch(A)$. Then by an \emph{edge} of $\ch(A)$ we mean a closed
segment $I$ of the straight line connecting two points of the
unit circle such that $I$ is contained in the boundary $\bd(\ch(A))$
of $\ch(A)$. By an \emph{edge} of a $\sim$-class we mean an edge of
the convex hull of that class.

\begin{dfn}[Laminations and dynamics]\label{d:si-inv-lam}
A lamination $\sim$ is ($\si_d$-){\em in\-va\-riant} if:

\noindent (D1) $\sim$ is {\em forward invariant}: for a class
$\mathbf{g}$, the set $\si_d(\mathbf{g})$ is a class too;

\noindent (D2) for any $\sim$-class $\mathbf{g}$, the map
$\si_d:\mathbf{g}\to\si_d(\mathbf{g})$ extends to $\uc$ as an
orientation preserving covering map such that $\mathbf{g}$ is the full
preimage of $\si_d(\mathbf{g})$ under this covering map.

\end{dfn}

Again, if this does not cause ambiguity, we will simply talk about
\emph{invariant} laminations.

Definition~\ref{d:si-inv-lam} (D2) has an equivalent version. Given a
closed set $Q\subset \uc$, a (positively oriented) {\em hole} $(a, b)$
of $Q$ (or of $\ch(Q)$) is a component of $\uc\sm Q$. Then (D2) is
equivalent to the fact that for a $\sim$-class $\mathbf{g}$ either
$\si_d(\mathbf{g})$ is a point or for each positively oriented hole
$(a, b)$ of $\mathbf{g}$ the positively oriented arc $(\si_d(a),
\si_d(b))$ is a hole of $\si_d(\mathbf{g})$. From now on, we assume
that, unless stated otherwise, $\sim$ is a $\si_d$-invariant
lamination.

\begin{figure}[!htb]
    \centering
    \begin{minipage}{0.5\textwidth}
        \centering
        \includegraphics[width=0.7\linewidth, height=0.2\textheight]{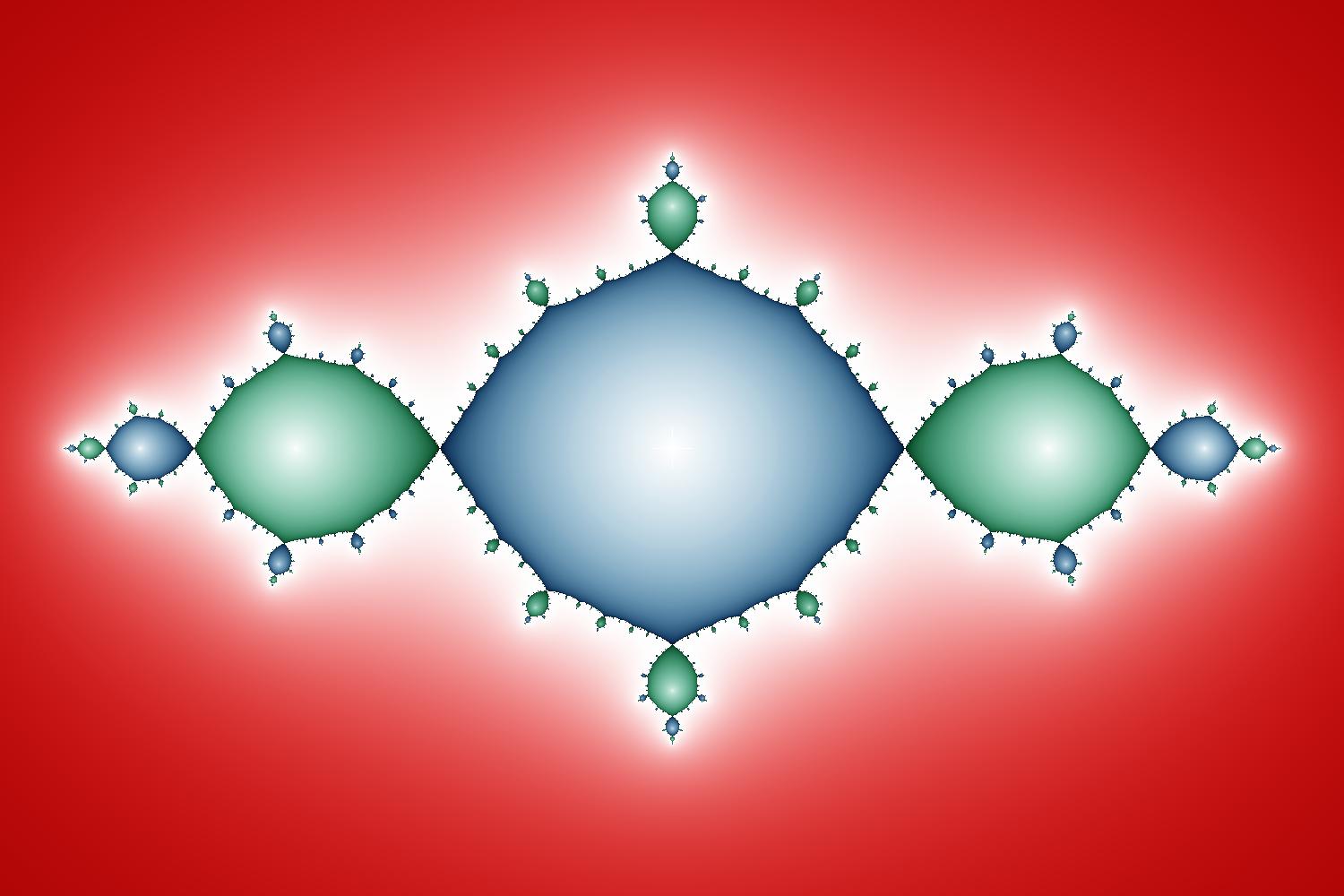}
        \caption{The Julia set of $f(z)=z^2-1$ (so-called ``basilica'')}
        \label{fig:basil}
    \end{minipage}%
    \begin{minipage}{0.5\textwidth}
        \centering
        \includegraphics[width=0.7\linewidth, height=0.2\textheight]{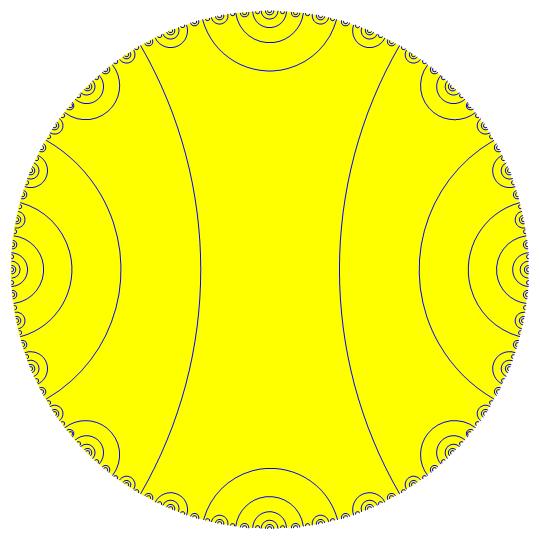}
        \caption{The (geo)lamination for the Julia set of $z^2-1$}
        \label{fig:qml}
    \end{minipage}
\end{figure}

Given $\sim$, consider the \emph{topological Julia set}
$\uc/\sim=J_\sim$ and the \emph{topological polynomial}
$f_\sim:J_\sim\to J_\sim$ induced by $\si_d$. Using Moore's Theorem,
embed $J_\sim$ into $\C$ and extend the quotient map
$\psi_\sim:\uc\to J_\sim$ to $\C$ with the only non-trivial fibers
being the convex hulls of non-degenerate $\sim$-classes. A
\emph{Fatou domain} of $J_\sim$ (or of $f_\sim$) is a bounded
component of $\C\sm J_\sim$. If $U$ is a periodic Fatou domain of
$f_\sim$ of period $n$, then $f^n_\sim|_{\bd(U)}$ is either
conjugate to an irrational rotation of $\uc$ or to $\si_k$ with some
$1<k\le d$ \cite{bl02}. In the case of irrational rotation, $U$ is
called a \emph{Siegel domain}.
The complement of the unbounded component of $\C\sm J_\sim$ is called
the \emph{filled-in topological Julia set} and is denoted by $K_\sim$.
Equivalently, $K_\sim$ is the union of $J_\sim$ and its bounded Fatou
domains. If the lamination $\sim$ is fixed, we may omit $\sim$ from the
notation. By default, we consider $f_\sim$ as a self-mapping of
$J_\sim$. For a collection $\mathcal R$ of sets,
denote the union of all sets from $\mathcal R$ by $\mathcal R^+$.

\begin{dfn}[Leaves]\label{d:geola}
If $A$ is a $\sim$-class, call an edge $\ol{ab}$ of $\ch(A)$ a
\emph{leaf (of $\sim$)}. All points of $\uc$ are also called
(\emph{degenerate\emph{)} leaves (of $\sim$)}.
\end{dfn}

The family of all leaves of $\sim$ is closed (the limit of a sequence
of leaves of $\sim$ is a leaf of $\sim$); the union of all leaves of
$\sim$ is a continuum. Extend $\si_d$ (keeping the notation) linearly
over all \emph{individual leaves of $\sim$} in $\ol{\disk}$. In other words,
for each leaf of $\sim$ we define its own linear map; note that in the end even though the extended $\si_d$ is not
well defined on the entire disk, it is well defined on the union of all
leaves of $\sim$.

\subsection{Geometric laminations}\label{ss:geol}

The connection between laminations, understood as equivalence
relations, and the original approach of Thur\-ston's \cite{thu85}
can be explained once we introduce a few key notions. Assume that
$\sim$ is a $\si_d$-invariant lamination. Thurston studied
collections of chords in $\disk$ similar to collections of leaves of
$\sim$ with no equivalence relation given.

\begin{dfn}[Geometric laminations, \rm{cf.} \cite{thu85}]\label{d:geolam}
Two distinct chords in $\cdisk$ are said to be \emph{unlinked} if they
meet at most in a common endpoint; otherwise they are said to be
\emph{linked}, or to \emph{cross} each other. A \emph{geometric
pre-la\-mi\-na\-tion} $\lam$ is a set of (possibly degenerate) chords
in $\ol{\disk}$ such that any two distinct chords from $\lam$ are
unlinked; $\lam$ is called a \emph{geolamination} if all points of
$\uc$ are elements of $\lam$, and $\lam^+$ is closed. Elements of
$\lam$ are called \emph{leaves} of $\lam$. By a \emph{degenerate} leaf
(chord) we mean a singleton in $\uc$. The continuum $\lam^+\subset
\cdisk$ is called the \emph{solid} (of $\lam$).
\end{dfn}

Important objects related to a geolamination are its gaps.

\begin{dfn}[Gaps]\label{d:gaps-s}
Let $\lam$ be a geolamination. The closure in $\C$ of a non-empty
component of $\disk\sm \lam^+$ is called a \emph{gap} of $\lam$.
If $G$ is a gap or a leaf, call the set $G'=\uc\cap G$ the \emph{basis
of $G$}. A gap is said to be \emph{finite $($infinite, countable,
uncountable$)$} if its basis is finite (infinite, countable,
uncountable). Uncountable gaps are also called \emph{Fatou gaps}.
Points of $G'$ are called \emph{vertices} of $G$.
\end{dfn}

Now let us discuss geolaminations in the dynamical context. A chord
(e.g., leaf of a (geo)\-la\-mi\-nation) is called
\emph{($\si_d$-)critical} if its endpoints have the same image under
$\si_d$. If it does not cause ambiguity, we will simply talk about \emph{critical} chords.
Definition~\ref{d:geolaminv} was introduced in \cite{thu85}.

\begin{dfn}[Invariant geolaminations in the sense of Thurston]\label{d:geolaminv}
A geolamination $\lam$ is said to be \emph{($\si_d$-)invariant in
the sense of Thurston} if the following conditions are satisfied:

\begin{enumerate}

\item (Leaf invariance) For each leaf $\ell\in \lam$, the set
    $\si_d(\ell)$ is a leaf in $\lam$ (if $\ell$ is critical, then
    $\si_d(\ell)$ is degenerate). For a non-degenerate leaf
    $\ell\in\lam$, there are $d$ pairwise disjoint leaves
    $\ell_1,\dots,\ell_d\in\lam$ with $\si_d(\ell_i)=\ell, 1\le
    i\le d$.

\item (Gap invariance) For a gap $G$ of $\lam$, the set
    $H=\ch(\si_d(G'))$ is a point, a leaf, or a gap of $\lam$, in which case
    $\si_d:\bd(G)\to \bd(H)$ is a positively oriented composition
    of a monotone map and a covering map (thus if $G$ is a gap with
    finitely  many edges all of which are critical, then its image
    is a singleton).
\end{enumerate}

\end{dfn}

Again, if it does not cause ambiguity we will simply talk about geolaminations
which are \emph{invariant} in the sense of Thurston.

We will use a special extension $\si^*_{d, \lam}=\si_d^*$ of $\si_d$ to
the closed unit disk associated with $\lam$. On $\uc$ and all leaves of
$\lam$, we set $\si^*_d=\si_d$. Define $\si^*_d$ on the interiors of
gaps using a standard barycentric construction \cite{thu85}. For
brevity, we sometimes write $\si_d$ instead of $\si^*_d$.

A useful fact about ($\si_d$-invariant) geolaminations in the sense
of Thurston is that we can define a topology on their family by
identifying each ($\si_d$-invariant) geolamination $\lam$ with its
laminational solid $\lam^+$ and using the Hausdorff metric on this
family of laminational solids. This produces a compact
metric space of laminational solids.

The most natural situation, in which $\si_d$-invariant geolaminations in
the sense of Thurston appear deals with $\si_d$-invariant laminations.

\begin{dfn}\label{d:q-geolam}
Suppose that $\sim$ is a ($\si_d$-invariant) lamination. The family
$\lam_\sim$ of all leaves of $\sim$ is called the \emph{geolamination
generated by $\sim$} or a \emph{($\si_d$-invariant) q-geolamination}.
\end{dfn}

The example below shows that not all  invariant geolaminations in the sense
of Thurston are sibling invariant.
Suppose that points $\hat x_1, \hat y_1, \hat z_1, \hat x_2, \hat
y_2, \hat z_2$ are positively ordered on $\uc$ and $H=\ch(\hat x_1
\hat y_1 \hat z_1 \hat x_2 \hat y_2 \hat z_2)$ is a critical hexagon
of an invariant q-geolamination $\lam$ such that $\si^*_d:H\to T$ maps
$H$ in the $2$-to-$1$ fashion onto the triangle $T=\ch(xyz)$ with
$\si_d(\hat x_i)=x, \si_d(\hat y_i)=y, \si_d(\hat z_i)=z$. Now, add
to the geolamination $\lam$ the leaves $\ol{\hat x_1 \hat z_1}$ and
$\ol{\hat x_1 \hat x_2}$ and all their pullbacks along the backward
orbit of $H$ under $\si^*_d$. Denote the thus created geolamination
$\lam'$. It is easy to see that $\lam'$ is invariant in the sense of Thurston
but not sibling invariant because $\ol{\hat x_1 \hat z_1}=\ell$
cannot be completed to a full sibling collection (clearly, $H$ does
not contain siblings of $\ell$).


\begin{figure}[h!]
\centering\def\svgwidth{.7\columnwidth}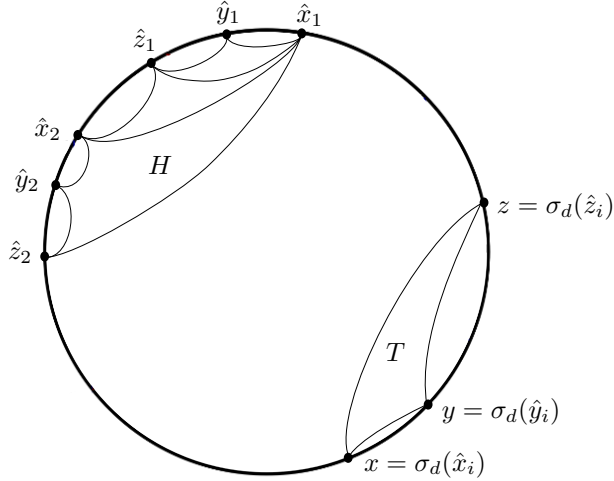
 \caption{An example of a  geolamination invariant in the sense of Thurston which is not sibling invariant.
 The leaf $\ol{\hat x_1 \hat z_1}$ has no siblings in $H$.}
 \label{nosibf}
\end{figure}

Another example can be given in the case of quadratic (i.e.,
$\si_2$-invariant) geolaminations; since it will be used in what
follows we describe it separately.

\begin{figure}[h]\label{notq}
   \begin{tikzpicture}
\draw[solid,ultra thick] (0,0) circle (4);
\draw[very thick] (-4,0) -- (4,0);

\node at (0,1) {\huge G};
\node at (0,-1) {\huge G$'$};

\draw[ultra thick] (-4,0) arc (270:360:4);
\draw[ultra thick] (4,0) arc (90:180:4);
\node at (-4.5,0) { $\frac{1}{2}$};
\node at (4.5, 0) { $0$};

\draw[ultra thick] (0,4) arc (180:315:1.65);
\draw[ultra thick] (0,-4) arc (0:135:1.65);
\node at (0,4.5) { $\frac{1}{4}$};
\node at (0,-4.5) { $\frac{3}{4}$};

\draw[ultra thick] (2.828,2.828) arc (180:247:1.5);
\draw[ultra thick] (-2.828,-2.828) arc (0:67:1.5);
\node at (3.3,3.3) { $\frac{1}{8}$};
\node at (-3.3,-3.3) { $\frac{5}{8}$};

\draw[ultra thick] (-4,0) arc (270:404:1.7);
\draw[ultra thick] (4,0) arc (90:224:1.7);
\node at (3.3,-3.3) { $\frac{3}{8}$};
\node at (-3.3,3.3) { $\frac{7}{8}$};

\draw[ultra thick] (3.72,1.48) arc (180:212:1.5);
\draw[ultra thick] (-3.72,-1.48) arc (0:32:1.5);
\node at (4.2,1.6) { $\frac{1}{16}$};
\node at (-4.2,-1.6) { $\frac{9}{16}$};

\end{tikzpicture}

    \caption{An example of a geolamination which is not a q-geolamination}
      \label{cate-fin}
\end{figure}
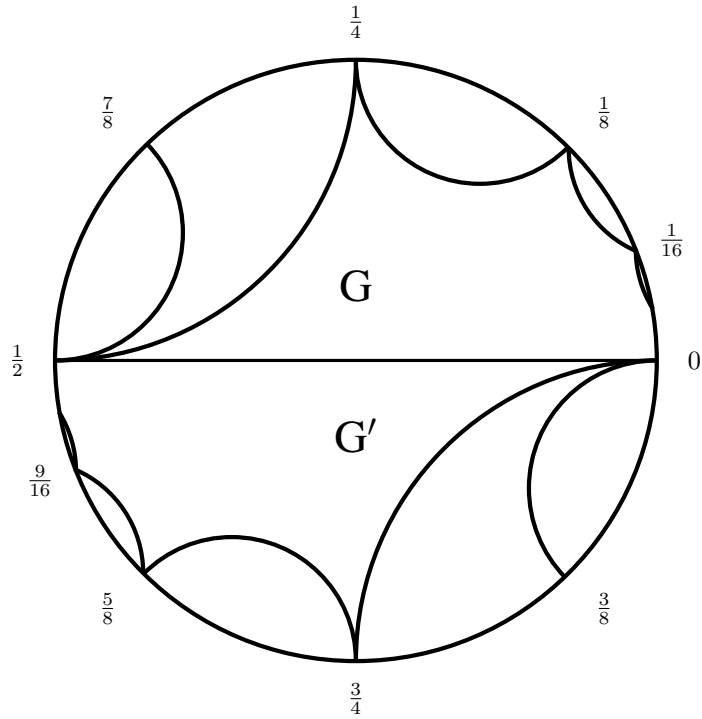

\begin{exa}\label{e:pull12}
Construct a quadratic geolamination $\lam_{12}$ as follows. Start
$\lam_{12}$ with the chord $\ol{0\frac12}=\di$ (this is the first
step in the construction). Then, on the second step, add to
$\lam_{12}$ one first pullback of $\di$ under $\si_2$ unlinked with
$\di$, namely, the chord $\ol{\frac14\frac12}$. Then, on the third
step, add one first pullback of $\ol{\frac14\frac12}$, namely,
$\ol{\frac18\frac14}$. Continue this construction so that on the
$n$-th step we add the leaf $\ol{2^{-n} 2^{-n+1}}$. This creates a
gap $G$ located above $\di$. Then construct a \emph{sibling} gap
$G'$ of $G$ by rotating $G$ by the angle $\pi$. From this moment on
we can pull back $G$ and $G'$ (or, equivalently, their edges)
choosing for each leaf its uniquely defined two pullbacks which do
not cross the already constructed leaves (the pullback construction
of an invariant geolamination with given critical leaves is actually
due to Thurston and is used a lot in his preprint \cite{thu85}). In
the end we will obtain an invariant geolaminations in the sense of
Thurston. It is easy to see that this is not a q-geolamination. Gaps
similar to $G$ are actually called \emph{caterpillar} gaps (see
Definition~\ref{d:caterpillar} given below).
\end{exa}

It is easy to see that our terminology is consistent in the sense that
($\si_d$-invariant) q-geolaminations satisfy properties from
Definitions~\ref{d:geolam} and \ref{d:geolaminv} and are, therefore,
($\si_d$-invariant) geolaminations. In what follows we talk
interchangeably about leaves (gaps) of $\sim$ or leaves (gaps) of
$\lam_\sim$.

\begin{dfn}[Critical gaps]\label{d:crit}
A gap\, $G$ of a (geo)lamination is said to be \newline
\emph{($\si_d$-)critical} if for each $y\in \si_d(G')$ the set
$\si_d^{-1}(y)\cap G'$ consists of at least $2$ points.
\end{dfn}

If it does not cause ambiguity, we will simply talk about \emph{critical}
gaps.

Slightly abusing the language, we will call the family of all degree
$d$ polynomials with connected Julia sets the \emph{connectedness
locus of degree $d$}. As was explained above, $\si_d$-invariant
laminations are naturally related to polynomials with locally
connected Julia sets.

However this leaves us with a problem of associating laminations (or
related objects) to polynomials $P$ whose Julia sets are not locally
connected. As one of the central problems in complex dynamics is
studying polynomials with connected Julia sets, in our paper we
consider this issue only for such polynomials. A natural approach here
is as follows. Suppose that $P$ is a polynomial of degree $d$ with
connected but not locally connected Julia set. Consider a sequence of
polynomials $P_i\to P$ with locally connected Julia sets $J_{P_i}$
 (this is always possible in the quadratic case). As we explained above, such polynomials
$P_i$ generate their canonical laminations $\sim_{P_i}$. One can hope
to use the appropriately designed limit transition and thus to define
the limit lamination $\sim_P$ associated with $P$.

To this end we consider q-geolaminations $\lam_{P_i}$ and associated
with them continua $\lam^+_{P_i}$. Assume that these continua converge
to a continuum in $\cdisk$, which itself is the union of pairwise
unlinked chords in $\cdisk$. These chords form a family of chords, which
turns out to be a $\si_d$-invariant geolamination (see
Definition~\ref{d:geolaminv}). However, this geolamination cannot be
associated with $P$ in a canonical way as we assumed that $J_P$ is
\emph{not} locally connected. Also, depending on the polynomial $P$ it
may happen that more than one limit geolamination can be associated
with $P$ as above. Nevertheless, at least in the quadratic case there
are only \emph{finitely many} such limit geolaminations. Denote this
\emph{finite} collection
of limit geolaminations by $\mathfrak L(P)$.

The idea is to associate to the collection $\mathfrak L(P)$ a \emph{unique} lamination
$\sim_P$ and its q-geolamination $\lam_{\sim_P}$ and to declare them
generated by $P$. Thus, to each polynomial $P$ we associate two
objects: a collection $\mathfrak L(P)$ of its limit geolaminations, and, on the other
hand, the corresponding lamination $\sim_P$ uniquely associated with
the collection $\mathfrak L(P)$. If now we consider these collections
$\mathfrak L(P)$ of limit
geolaminations as classes of equivalence and factor the closure of
the space of all q-geolaminations accordingly, we get a quotient space. This
quotient space topologizes the set of all quadratic laminations (or,
equivalently, the set of all quadratic topological polynomials).

To sum it all up, one can say, that from the point of view of
polynomials only laminations and q-geolaminations are important. It
follows that to understand the structure of the connectedness locus
of degree $d$ it is natural to study the closure of all
$\si_d$-invariant laminations. Thus, we need to define a suitable
topology on the family of all $\si_d$-invariant laminations, which
would reflect the topology of the connectedness locus.

Our approach to this problem is as follows. First, following Thurston
we associate to each $\si_d$-invariant lamination $\sim$ its
q-geolamination $\lam_\sim$. To define a suitable topology on the
family of all $\si_d$-invariant q-geolaminations one can identify this
family with the family of their laminational solids endowed with the
Hausdorff metric.
However,\, taken ``as is'' the resulting topological space cannot serve
its purpose because the limit of laminational solids of
\emph{q-geolaminations} often is not a laminational solid of a
\emph{q-geolamination} (even though by the remark above this limit is a
laminational solid of some \emph{$\si_d$-invariant geolamination}).
Thus, even if a sequence of q-geolaminations is such that their solids
converge to a solid of a $\si_d$-invariant geolamination, we cannot
directly associate a q-geolamination to this limit. This justifies our
study of \emph{limit geolaminations} (more precisely, of geolaminations
that are limits of $\si_d$-invariant q-geolaminations), which is done
for an arbitrary degree $d$ in Section~\ref{s:limitg}.

We overcome the obstacle just described in the case when $d=2$ (we call all
$\si_2$-invariant laminations and geolaminations \emph{quadratic}) as
follows. Take all q-geolaminations and limits of their laminational
solids. The resulting compact metric space of laminational solids is
then factored according to a specific natural equivalence related to
studying and comparing critical sets of quadratic q-geolaminations and
their limits. In this way we construct the appropriate quotient space
of the space of laminational solids of all quadratic q-geolaminations
and their limits in the Hausdorff metric (using our identifications we
can also talk about the quotient space of the space of all quadratic
q-geolaminations and their limits). We then prove that this quotient
space is homeomorphic to the combinatorial Mandelbrot set $\Mcc$.

To implement our program we will work with so-called \emph{sibling
($\si_d$-in\-va\-riant) geolaminations}. They form a closed subspace
of the space of all $\si_d$-invariant geolaminations, which still
contains all q-geolaminations (in other words, q-geolaminations and
all their limits are sibling $\si_d$-invariant). Since our main
interest lies in studying q-geolaminations and their limits, it will
be more convenient to work with sibling $\si_d$-invariant
geolaminations than with $\si_d$-invariant geolaminations in the
sense of Thurston. Other advantages of working with sibling
$\si_d$-invariant geolaminations are that they are defined through
properties of their leaves (gaps are not involved in the definition)
and that the space of all of them is smaller (and hence easier to
deal with) than the space of all $\si_d$-invariant geolaminations.

\begin{dfn}\label{d:siblinv}
A geolamination $\lam$ is \emph{sibling ($\si_d$-)invariant}
provided:
\begin{enumerate}
\item for each $\ell\in\lam$, we have $\si_d(\ell)\in\lam$,
\item \label{2}for each $\ell\in\lam$ there exists $\ell'\in\lam$ so that $\si_d(\ell')=\ell$.
\item \label{3} for each $\ell\in\lam$ so that $\si_d(\ell)=\ell'$
    is a non-degenerate leaf, there exist d {\bf disjoint} leaves
    $\ell_1,\dots,\ell_d$ in $\lam$ so that $\ell=\ell_1$ and
    $\si_d(\ell_i)=\ell'$ for all $i=1,\dots,d$.
\end{enumerate}
\end{dfn}

Let us list a few properties of sibling $\si_d$-invariant
geolaminations.

\begin{thm}[\cite{bmov13}]\label{t:siblinv} Sibling $\si_d$-invariant
geolaminations are invariant in the sense of Thurston. The space of
all sibling $\si_d$-invariant geolaminations is compact. All
geolaminations generated by $\si_d$-invariant laminations are
sibling $\si_d$-invariant.
\end{thm}

While by Theorem~\ref{t:siblinv} all sibling $\si_d$-invariant
geolaminations are invariant in the sense of Thurston, the opposite is
not true already for quadratic geolaminations. Indeed, consider
Example~\ref{e:pull12} and the geolamination $\lam_{12}$ suggested
there. Amend $\lam_{12}$ by removing the chord $\ol{\frac12\frac34}$
and all its pullbacks. Then the amended geolaminations $\lam^a_{12}$
remains Thurston invariant, however it is no loner sibling
$\si_2$-invariant. Indeed, a leaf $\ol{0\frac14}\in \lam^1_{12}$ does
not have a \emph{disjoint} sibling leaf as the only sibling leaf it has
is the leaf $\ol{\frac14\frac12}$ which is non-disjoint from
$\ol{0\frac14}$.

Let us now discuss gaps in the context of $\si_d$-invariant
(geo)laminations.

\begin{dfn}[Periodic and (pre)periodic gaps]\label{d:gaps-i}
Let $G$ be a gap of an invariant geolamination $\lam$.
If the map $\si_d$ restricted on $G'$ extends to $\bd(G)$ as a
composition of a monotone map and a covering map of some degree $m$,
then $m$ is called the \emph{degree} of $\si_d|_G$. A gap/leaf $U$ of
$\lam_\sim$ is said to be \emph{{\rm(}pre{\rm)}periodic} of period $k$
if $\si_d^{m+k}(U')=\si_d^m(U')$ for some $m\ge 0$, $k>0$; if $m, k$
are chosen to be minimal, then if $m>0$, $U$ is called
\emph{preperiodic}, and, if $m=0$, then $U$ is called \emph{periodic
$($of period $k)$}. If the period of $G$ is $1$, then $G$ is said to be
\emph{invariant}. We define \emph{precritical} and
\emph{{\rm(}pre{\rm)}critical} objects similarly to how (pre)periodic
and preperiodic objects are defined above.
\end{dfn}

A more refined series of definitions deals with infinite periodic
gaps of sibling $\si_d$-invariant (geo)laminations. There are three
types of such gaps: \emph{caterpillar} gaps, \emph{Siegel} gaps, and
\emph{Fatou gaps of degree greater than one}. We define them below.
Observe that by \cite{kiw02} infinite gaps eventually map onto
periodic infinite gaps. First we state (without a proof) a very-well
known folklore lemma about edges of preperiodic (in particular,
infinite) gaps.

\begin{lem}\label{l:edges} Any edge of a (pre)periodic gap is either
(pre)periodic or (pre)critical.
\end{lem}

Let us now classify infinite gaps.

\begin{dfn}\label{d:caterpillar} An infinite  gap $G$ is said
to be a \emph{caterpillar} gap if its basis $G'$ is countable
(see Figure~6).
\end{dfn}

As as an example, consider a periodic gap $Q$ such that:
\begin{itemize}
 \item  The boundary of $Q$ consists of a periodic leaf
     $\ell_0=\ol{xy}$ of period $k$, a critical leaf
     $\ell_{-1}=\ol{yz}$ concatenated to it, and a countable
     concatenation of leaves $\ell_{-n}$ accumulating at $x$ (the
     leaf $\ell_{-r-1}$ is concatenated to the leaf $\ell_{-r}$,
     for every $r=1$, 2, $\dots$).
\item We have $\si^k(x)=x$, $\si^k(\{y, z\})=\{y\}$, and $\si^k$
    maps each $\ell_{-r-1}$ to $\ell_{-r}$ (all leaves are shifted
    by one towards $\ell_0$ except for $\ell_0$, which maps to
    itself, and $\ell_{-1}$, which collapses to the point $y$).
\end{itemize}

The description of $\si_3$-invariant caterpillar gaps is in
\cite{bopt13}. In general, the fact that the basis $G'$ of a
caterpillar gap $G$ is countable implies that there are lots of
concatenated edges of $G$. Other properties of caterpillar gaps can
be found in Lemma~\ref{l:cater}.

\begin{lem}\label{l:cater}
Let $G$ be a caterpillar gap of period $k$. Then the degree of
$\si_d^k|_G$ is one and $G'$ contains some periodic points.
\end{lem}

\begin{proof}
We may assume that $k=1$. Consider $\si_d|_{\bd(G)}$. It is
well-known that if the degree $r$ of $\si_d|_{\bd(G)}$ is greater
than one, then there is a monotone map $\psi:\bd(G)\to \uc$ that
semiconjugates $\si_d|_{\bd(G)}$ and $\si_r|_{\uc}$ (see, e.g.,
\cite{blo86, blo87a, blo87b} where a similar claim is proven for
``graph'' maps). Take the set $B$ all points of $\bd(G)$ that do
not belong to open segments in $\bd(G)$, on which $\psi$ is
constant (such sets are said to be \emph{basic} in \cite{blo86,
blo87a, blo87b}).

Edges of $G$ must be collapsed to points under $\psi$ because otherwise
their $\psi$-images would have an eventual $\si_d$-image covering the
whole $\uc$ while by Lemma~\ref{l:edges} any edge of $G$ eventually
maps to a point or to a periodic edge of $G$ and cannot have the image
that is so big. Since $B$ is clearly
uncountable, we get a contradiction.

If now the degree of $\si_d|_{\bd(G)}$ is one, then it is well known
\cite{ak79, blo84} that either (1)  $\si_d|_{\bd(G)}$ is
monotonically semiconjugate to an irrational rotation by a map
$\psi$, or (2) $\si_d|_{\bd(G)}$ has periodic points. Take the set
$B$ of all points of $\bd(G)$ that do not belong to open segments
in $\bd(G)$, on which $\psi$ is a constant. If case (1), then,
similarly to the above, the edges of $G$ must be collapsed to points
under $\psi$ because otherwise there would exist a finite union of
their $\psi$-images covering the whole $\uc$ while by
Lemma~\ref{l:edges} any edge of $G$ eventually maps to a point or to
a periodic edge of $G$. Hence $B$ is uncountable contradicting the
definition of a caterpillar gap. Thus, (2) holds.
\end{proof}

\begin{dfn}\label{d:siegel}
A periodic Fatou gap $G$ of period $n$ is said to be a periodic
\emph{Siegel} gap if the degree of $\si_d^n|_G$ is $1$ and the basis
$G'$ of $G$ is uncountable.
\end{dfn}

The next lemma is well-known, a part of it  was actually proven in the proof of
Lemma~\ref{l:cater}.

\begin{lem}\label{l:siegel}
Let $G$ be a Siegel gap of period $n$. Then the map $\si_d^n|_{\bd(G)}$
is monotonically semiconjugate to an irrational circle rotation and
contains no periodic points. A periodic Siegel gap must have at least
one image that has a critical edge.
\end{lem}

The following definition completes our series of definitions.

\begin{dfn}\label{d:fatou}
A periodic \emph{Fatou gap is of degree $k>1$} if the degree of
$\si_d^n|_{\bd(G)}$ is $k>1$. If the degree of a Fatou gap $G$ is $2$,
then $G$ is said to be \emph{quadratic}.
\end{dfn}

The next lemma is well-known.

\begin{lem}\label{l:fatou}
Let $G$ be a Fatou gap of period $n$ and of degree $k>1$. Then the map
$\si_d^n|_{\bd(G)}$ is monotonically semiconjugate to $\si_k$.
\end{lem}

\section{Limit geolaminations and their properties}\label{s:limitg}

In this section we study properties of limits of $\si_d$-invariant
q-geolaminations (as explained above, convergence of geolaminations is
understood as convergence of their laminational solids in the Hausdorff metric).
Fix the degree
$d$. We prove a few lemmas, in which we assume that a sequence of
$\si_d$-invariant q-geolaminations $\lam_i$ converges to a sibling
$\si_d$-invariant geolamination $\lam_\infty$. By an \emph{(open)
strip} we mean a part of the unit disk contained between two disjoint
chords. By an \emph{(open) strip around a chord $\ell$} we mean a strip
containing $\ell$. In what follows when talking about convergence of
leaves/gaps, closeness of leaves/gap, and closures of families of
geolaminations we always mean this in the Hausdorff metric on the space
of their laminational solids.

\begin{dfn}\label{d:closuq}
Let $\laq_d$ be the family of all $\si_d$-invariant q-geolaminations.
Let $\ol{\laq_d}$ be the closure of $\laq_d$ in the compact space of all subcontinua of
$\ol{\disk}$ with the Hausdorff metric: we
take the closure of the family of laminational solids of geolaminations
from $\laq$, for each limit continuum consider the corresponding
geolamination, and denote the family of all such geolaminations by
$\ol{\laq_d}$.
\end{dfn}

Even though we will prove below a few general results, we mostly
concentrate upon studying periodic objects of limits of
$\si_d$-invariant q-geolaminations.

\begin{lem}\label{l:limleaf1}
Let $\lam\in \ol{\laq_d}$, $\ell=\ol{ab}$ be a periodic leaf of $\lam$.
If $\hlam\in \laq_d$ is sufficiently close to $\lam$, then any leaf of
any $\hlam$ sufficiently close to $\ell$ is either equal to $\ell$ or
disjoint from $\ell$. Moreover, if $\lam_i\to \lam, \lam_i\in \laq_d$,
then for any $\e>0$ there is $N=N(\e)$ such that any leaf of $\lam_i$
$(i>N)$ is disjoint from $\ell$ or intersects $\ell$ at a point $z$
that is $\e$-close to $\{a,b\}$.
\end{lem}

\begin{proof}
If a leaf $\hell\ne \ell$ is a leaf of a q-geolamination that is
very close to $\ell$ and non-disjoint from
$\ell$, then it must cross $\ell$ (if $\hell$ shares an endpoint with
$\ell$, then the other endpoint of $\hell$ must be periodic of the same
period as $a$ and $b$ and hence $\hell$ cannot be close to $\ell$).
However, then it would follow that $\si_d^n(\hell)$ crosses $\hell$ (here
$n$ is such that $\si_d^n(a)=a, \si_d^n(b)=b$), a contradiction. This
proves the first claim.

Now, let $\lam_i\to \lam, \lam_i\in \laq_d$. Choose $N=N(\e)$ so
that each geolamination $\lam_i, i>N$ has a leaf $\ell_i$ that is
much closer to $\ell$ than $\e$. If $\lam_i$ has a leaf $\hell_i\ne
\ell$ intersecting $\ell$ at a point $z$ with $\mathrm{d}(z,
\{a, b\})\ge \e$, then, since $\hell_i$ does not cross $\ell_i$, the
leaf $\hell$ will be close to $\ell$, contradicting the above.
\end{proof}

Definition~\ref{d:rigileaf} introduces the
concept of rigidity.

\begin{dfn}\label{d:rigileaf} A leaf/gap $G$ of $\lam$ is \emph{rigid}
if any q-geolamination close to $\lam$ has $G$ as its leaf/gap.
\end{dfn}

Periodic leaves of geolaminations are either edges of gaps or
limits of other leaves; consider these two cases separately.

\begin{lem}\label{l:limleaf2}
Let $\lam\in \ol{\laq_d}$, and let $\ell=\ol{ab}$ be a periodic leaf of $\lam$
that is not an edge of a gap of $\lam$. Then $\ell$ is rigid.
\end{lem}

\begin{proof}
By the assumption, arbitrarily close to $\ell$ on either side of
$\ell$ there are leaves $\ell^l\ne \ell$ and $\ell^r\ne \ell$.
Observe that leaves $\ell^l$ and $\ell^r$ may share an endpoint with
$\ell$, still either leaf has at least one endpoint on the
appropriate side of $\ell$. Choose them very close to $\ell$. Now,
choosing a q-geolamination $\hlam\in \laq_d$ very close to $\lam$ we
may choose leaves $\hell^l\in \hlam$ and $\hell^r\in \hlam$ very
close to $\ell^l$ and $\ell^r$. Since $\ell^l$ and $\ell^r$ are very
close to $\ell$, by Lemma~\ref{l:limleaf1} the leaves $\hell^l,
\hell^r$ either coincide with $\ell$ or are disjoint from $\ell$.
Since $\ell^l\ne \ell$ and $\ell^r\ne \ell$, we have $\hell^l\ne \ell$
and $\hell^r\ne \ell$. Thus, the leaves $\hell^l$ and $\hell^r$ are
disjoint from $\ell$. This and the choice of the leaves $\ell^l$ and
$\ell^r$ implies that $\hell^l$ and $\hell^r$ are the edges of a
narrow strip $\widehat S$ around $\ell$. Choose $n$ so that
$\si^n_d(a)=a$ and $\si_d^n(b)=b$. Then $\hlam$ has a pullback-leaf
of $\hell^l$ inside $\widehat S$ whose endpoints are even closer to
$a$ and $b$. Repeating this, we see that $\ell$ is a leaf of
$\hlam$. Thus, $\ell$ is rigid.
\end{proof}

To study periodic edges of gaps we use Lemma~\ref{l:limgap1}, which
is straightforward and the proof of which is left to the reader.
Let $\lam\in \ol{\laq_d}$, and let $G$ be a gap of $\lam$. Then
for a geolamination $\hlam$ define $G(\hlam)$ as the gap of $\hlam$
with the area of $G(\hlam)\cap G$ greater than half\, the area of\, $G$
(if such a gap exists) or $\0$ otherwise. Clearly, $G(\hlam)$ is
well-defined for any $\hlam$.

\begin{lem}\label{l:limgap1}
Let $\lam\in \ol{\laq_d}$, and let $G$ be a gap of $\lam$. Then for any
geolamination $\hlam\in \laq_d$ close to $\lam$ the gap $G(\hlam)$ is
non-empty and such that $G(\hlam_i)\to G$ as $\hlam_i\to \lam$. Moreover,
if $\si_d^n(G)=G$ for some $n$, then $\si_d^n(G(\hlam))=G(\hlam)$ if
$\hlam$ is close to $\lam$.
\end{lem}

A periodic leaf of a geolamination $\lam\in \ol{\laq_d}$ that is an
edge of a gap has specific properties.

\begin{lem}\label{l:limgap1.5}
Let $\lam\in \ol{\laq_d}$, let $G$ be a gap of $\lam$, and let $\ell$ be a
periodic edge of $G$. Then for all geolaminations $\hlam\in \laq_d$
close to $\lam$ their gaps $G(\hlam)$ are such that either $\ell$ is an
edge of $G(\hlam)$, or $\ell$ intersects the interior of $G(\hlam)$ and
$G(\hlam)$ has an edge close to $\ell$. Moreover, the following holds:

\begin{enumerate}

\item If $G$ is periodic of period $n$, then $\ell$ must be an edge
    of $G(\hlam)$ and either $G$ is finite, or the degree of
    $\si_d^n|_{\bd(G)}$ is greater than one.

\item If $\ell$ is a common edge of gaps $G, H$ of $\lam$ then one
    of $G, H$ is a Fatou gap of degree greater than one and the
    other one is either a finite periodic gap, or a Fatou gap of
    degree greater than one.

\end{enumerate}

\end{lem}

\begin{proof}
Suppose that $\ell=\ol{ab}$ and that $\si_d^n(a)=a$, $\si_d^n(b)=b$.
If $\hlam$ is sufficiently close
to $\lam$, then gaps $G(\hlam)$ exist by Lemma~\ref{l:limgap1}. By
definition, $G(\hlam)$ has an edge $\hell$ close to $\ell$. If
$\hell\ne \ell$, then $\hell$ is disjoint from $\ell$, and $\ell$ is
either disjoint or non-disjoint from the interior of $G(\hlam)$. In
the former case $\si_d^n(\hell)$ intersects the interior of
$G(\hlam)$, a contradiction. Thus, if $\hell\ne \ell$, then $\ell$
intersects the interior of $G(\hlam)$ and $G(\hlam)$ has an edge
close to but disjoint from $\ell$.

Let $\si_d^n(G)=G$. If $\hlam$ is close to $\lam$, then by
Lemma~\ref{l:limgap1} the gaps $G(\hlam)$ are well-defined and such
that $\si_d^n(G(\hlam))=G(\hlam)$. If $\ell$ is not an edge of
$G(\hlam)$, then by way of contradiction we may choose the edge $\hell$
of $G(\hlam)$ with $\hell\cap \ell=\0, \hell \to \ell$ as $\hlam\to
\lam$ while also having that $\ell$ intersects the interior of
$G(\hlam)$. Yet, this would imply that $\si_d^n(G(\hlam))\ne G(\hlam)$,
a contradiction. Moreover, suppose that $G$ is not finite. Then
$G(\hlam)$ is not finite either. Indeed, $G(\hlam)$ is a periodic gap
of a q-geolamination $\hlam$, and $\ell$ is an edge of $G(\hlam)$. Hence
$G(\hlam)$ is either a finite gap (and then we may assume that it is
the same gap for all $\hlam$'s), or a periodic Fatou gap of degree
greater than one. On the other hand, $G(\hlam)\to G$ as $\hlam\to
\lam$. This implies that either $G$ is fixed and rigid, or, in the
limit, the degree of $\si_d^n|_{\bd(G)}$ is greater than one.

If $\ell$ is a common edge of two gaps $G, H$ of $\lam$, a
geolamination $\hlam\in \laq_d$ close to $\lam$ has gaps $G(\hlam),
H(\hlam)$ close to $G$ and $H$. By the above $G(\hlam)$ and $H(\hlam)$
must share the leaf $\ell$ as their\, edge, hence $\ell$ is rigid.
Since $\hlam\in \laq_d$, either these gaps are both periodic Fatou gaps
of degree greater than one or one of them is finite periodic and the
other one is a periodic Fatou gap of degree greater than one. Since
$\lam\in  \ol{\laq_d}$, there is a sequence of geolaminations
$\lam_i\to \lam, \lam_i\in \laq_d$. Hence $\lam$ must have gaps $G$ and
$H$ of the same types as desired.
\end{proof}

To study (pre)periodic leaves we need Lemma~\ref{l:limgap3}.

\begin{lem}\label{l:limgap3}
Let $\lam\in \laq_d,$ let $G$ be a gap of $\lam$, let $H=\si^k_d(G)$ be
a gap, and let $\hell$ be an edge of $H$ such that, for any
geolaminations $\lam_i\to \lam, \lam_i\in \laq_d$ and their gaps
$H_i\to H$, the leaf $\hell$ is an edge of $H_i$ for large $i$ (e.g.,
this holds if $\hell$ is a periodic edge of a periodic gap $H$). If
$\ell$ is an edge of $G$ with $\si^k_d(\ell)=\hell$, then for any
geolaminations $\lam_i\to \lam, \lam_i\in \laq_d$ and any sequence of
their gaps $G_i\to G$, the leaf $\ell$ is an edge of $G_i$ for large
$i$. Thus, $(1)$ a (pre)periodic leaf of a gap that eventually maps to
a periodic gap, is rigid, and $(2)$ a finite gap that eventually maps
onto a periodic gap, is rigid.
\end{lem}

\begin{proof}We use the notation introduced in the statement of the Lemma.
By way of contradiction let us assume that there is a sequence of
$\si_d$-invariant q-geolaminations $\lam_i\to \lam$ with gaps $H_i\to
H$ and $G_i\to G$ such that $\si_d^k(G_i)=H_i$ and $\ell$ is not an
edge of $G_i$ for all $i$ (while $\hell$ is an edge of all $H_i$). Then
we can always choose an edge $\ell_i$ of $G_i$ so that $\ell_i\to
\ell$. Then $\si_d^k(\ell_i)\to \hell$, and by the assumption
$\si_d^k(\ell_i)=\hell$ for large $i$. Since $\ell_i\to \ell$ this
actually implies that $\ell_i=\ell$ for large $i$ as desired.
\end{proof}

For completeness, let us show that in some cases rigidity of pullbacks
of rigid leaves can be proven regardless of periodicity. By a
\emph{polygon} we mean a finite convex polygon. By a
\emph{($\si_d$-)collapsing polygon} we mean a polygon $P$, whose edges
are chords of $\cdisk$ such that their $\si_d$-images are the
same non-degenerate chord (thus as we walk along the edges of $P$,
their $\si_d$-images walk back and forth along the same non-degenerate
chord; as before, if it does not cause ambiguity we simply talk about
\emph{collapsing polygons}). When we say that $Q$ is a \emph{collapsing polygon} of a
geolamination $\lam$, we mean that all edges of $Q$ are leaves of
$\lam$; we also say that $\lam$ \emph{contains a collapsing polygon Q}.
However, this does not imply that $Q$ is a gap of $\lam$ as $Q$ might
be further subdivided by leaves of $\lam$ inside $Q$.

\begin{lem}[Lemmas 3.11, 3.14 from \cite{bmov13}]\label{l:3-11-14} Let
$\lam$ be a sibling $\si_d$-invariant geolamination. Suppose that
$L=\ell_1\cup\dots \ell_k$ is a concatenation of leaves of $\lam$ such
that $\si_d(\ell_i)=\ell, 1\le i\le k,$ for some non-degenerate leaf
$\ell$. Then there exists a maximal collapsing polygon $P$ of $\lam$
such that $L\subset P$ and the $\si_d$-image of any edge of $P$ equals
$\ell$. Moreover, any leaf of $\lam$ whose image is $\ell$, is either
disjoint from $P$ or is contained in $P$.
\end{lem}

Often rigid leaves of a limit geolamination give rise to rigid
pullbacks.

\begin{lem}\label{l:rigipul}
Consider a lamination $\lam\in \ol{\laq_d}$, a non-degenerate leaf $\hell$ of $\lam$,
and a leaf $\ell$ of $\lam$ with $\si^k_d(\ell)=\hell$ for some $k\ge
0$. If $\hell$ is rigid and no leaf $\ell, \si_d(\ell), \dots,
\si^{k-1}(\ell)$ is contained in a collapsing polygon of $\lam$, then
$\ell$ is rigid.
\end{lem}

\begin{proof}
First we prove the lemma for $k=1$. By way of contradiction suppose
that the leaf $\ell$ is not rigid. Then we may choose a sequence of
$\si_d$-invariant q-geolaminations $\lam_i\to \lam$ such that $\ell$ is
\emph{not} a leaf of any of them. Since $\hell$ is rigid, we may assume
that $\hell$ is a leaf of all $\lam_i$. By properties of
$\si_d$-invariant q-geolaminations we may also assume that there is a
collection of $d$ pairwise disjoint leaves $\ell_1, \dots, \ell_d$, all distinct from $\ell$,
such that all these leaves belong to every $\lam_i$ and map to $\hell$ under
$\si_d$. Clearly, all leaves $\ell_i, 1\le i\le d$ also belong to $\lam$.
Thus, the two endpoints of $\ell$ are also endpoints of two leaves,
say, $\ell_i$ and $\ell_j$ of $\lam$. The chain $\ell_i\cup \ell\cup \ell_j$
satisfies the conditions of Lemma~\ref{l:3-11-14}. Hence, $\ell$ is
contained in a collapsing polygon of $\lam$, a contradiction. Induction
now proves the lemma for any $k\ge 1$.
\end{proof}

Now we study rigidity of infinite periodic gaps. Consider the
quadratic case. Suppose that $\ell\ne \ol{0 \frac12}$ is a diameter
of $\cdisk$ and denote by $A_\ell$ the closed semi-circle based upon
$\ell$ and not containing $0$. Let $S_\ell$ be the set of all points
of $\uc$ with entire orbits contained in $A_\ell$. It is known that
for an uncountable family of diameters $\ell$ the set $S_\ell$ is a
Cantor set containing the endpoints of $\ell$. Moreover, for these
diameters $\ell$ the map $\si_2$ restricted on $\bd(\ch(S_\ell))$ is
semiconjugate to an irrational rotation and the set $\ch(S_\ell)$
itself is called an invariant \emph{Siegel} gap. Call such diameters
$\ell$ \emph{Siegel diameters}.

It is easy to see that in fact for each Siegel diameter $\ell$ there
exists the unique quadratic lamination $\sim_\ell$, of which $\ch(S_\ell)$
is a unique invariant gap. In fact, if $\ell_i\to \ell$ is a sequence
of Siegel diameters converging to a Siegel diameter, then one can show
that $\lam_{\sim_{\ell_i}}\to \lam_{\sim_{\ell}}$. On the other hand,
 if $\ell\ne \hell$ are two distinct Siegel
diameters, then $S_\ell\ne S_{\hell}$. Thus, in this case Siegel gaps
are \emph{not} rigid. Observe that the Siegel gaps described above \emph{do
not} have periodic edges but \emph{do} have critical edges.

It turns out that presence of critical edges of periodic gaps is
necessary for their non-rigidity. Recall that, for a gap $G$, a
\emph{hole} of $G$ is an arc $(a, b)$  such that $\ol{ab}$ is an
edge of $G$ and $(a, b)$ contains no points of $G'$; this hole of
$G$ is said to be the \emph{hole of $G$ behind $\ol{ab}$} and is
denoted by $H_G(\ell)$.

\begin{lem}\label{l:rigigap1}
Suppose that $G$ is a periodic Fatou gap of a geolamination $\lam\in
\ol{\laq_d}$. If no image of $G$ has critical edges, then $G$ is rigid.
\end{lem}

\begin{proof}
Suppose that $G$ is of period $k$ and degree $r>1$ and that
no eventual image of $G$ has critical edges. Without loss of
generality we may assume that $k=1$. We need to show that if a
sequence of $\si_d$-invariant q-geolaminations is such that
$\lam_i\to \lam$, then for some $N$ and all $i>N$ the gap $G$ is a
gap of $\lam_i$. By Lemma~\ref{l:limgap3} for any (pre)periodic edge
$\ell$ of $G$ there is $N=N(\ell)$ such that $\ell$ is an edge of
$G_i$. Choose the set $\mathcal A$ of all edges $\ell$ of $G$ such
that the holes $H_G(\ell)$ are of length greater than or equal to
$\frac1d$. Then there are finitely many such edges of $G$. Moreover,
by the assumption there are no critical edges in $\mathcal A$
(because there are no critical edges of $G$ at all).

Set $A=\uc\sm \bigcup_{\ell\in \mathcal A} H_G(\ell)$. It is easy to
see that $G'$ is in fact the set of all points of the circle that have
their entire orbits contained in $A$. Indeed, it is obvious that all
points of $G'$ have their entire orbits contained in $A$. Now, take a
point $x\in A\sm G'$. Set $I=H_G(\ell)$ to be a hole of $G$ containing
$x$. Since $\si_d$ is expanding, for some minimal $n$ we will have that
$\si_d^n(\ell)\in \mathcal A$. At this moment $x$ will be mapped
outside $A$, which shows that $x$ does no belong to the set of all
points of the circle that have their entire orbits contained in $A$.
It follows that if $N$ is chosen so that, for any $i>N$, all edges of $G$
belonging to $\mathcal A$ are also edges of $G_i$, then $G_i=G$.
\end{proof}

Geolaminations $\lam$ that belong to the closure of the family of all
$\si_d$-invariant q-geolaminations admit a phenomenon, which is
impossible for q-geolaminations, namely, they might have more than two
leaves coming out of one point of the circle.

\begin{dfn}\label{d:cone}
A family $C$ of leaves $\ol{ab}$ sharing the same endpoint $a$
is said to be a \emph{cone (of leaves of $\lam$)}. The point $a$ is
called the \emph{vertex} of the cone $C$; the set $\uc\cap C^+$ is called
the \emph{basis} of the cone $C$ and is denoted by $C'$.
We will identify $C$ with $C^+$.
A cone is said
to be \emph{infinite} if it consists of infinitely many leaves.
\end{dfn}

A few initial general results about cones of sibling $\si_d$-invariant
geolaminations are obtained in \cite{bmov13}.

\begin{lem}[Corollary 3.17 \cite{bmov13}]\label{l:3.17}
Let $\lam$ be a sibling $\si_d$-invariant geolamination and $T\subset
\lam^+$ be a cone of $\lam$ consisting of two or three leaves with a
common endpoint $v$. Suppose that $S\subset \lam^+$ is a cone of $\lam$
with $\si_d(S)=T$ such that and $\si_d|_S$ is one-to-one. Then the
circular orientation of the sets $T'$ and $S'$ is the same.
\end{lem}

We are mostly interested in studying cones with periodic vertices
(without loss of generality we will actually consider cones with fixed
vertices). A trivial case here is that of a finite cone.

\begin{lem}\label{l:finicon}
Let $\lam$ be a sibling $\si_d$-invariant geolamination. If $\ol{ab}$
is a leaf of $\lam$ with periodic endpoints, then the periods of $a$ and
$b$ coincide. In particular, if $C$ is a finite cone of $\lam$ with a
periodic vertex, then all points of its basis $C'$ are of the same
period.
\end{lem}

\begin{proof}
Let $\ol{ab}$ be a leaf of $\lam$ with periodic endpoints. If $a$ is of
period $n$ while $b$ is period $m>n$, consider $\si_d^n|_C$. Then the
$\si_d^n$-orbit of $\ol{ab}$ is a finite cone with $\si_d^n$-fixed
vertex $a$ which consists of more than one leaf such that all its
leaves share the endpoint $a$ and are cyclically permuted. Since by
Lemma~\ref{l:3.17} the circular order in the basis of a cone is
preserved under $\si^n_d$, we obtain a contradiction. This proves the
first claim of the lemma.

To prove the second, suppose that $C$ is a finite cone of $\lam$ with a
fixed vertex $v$ that has a non-periodic leaf $\ol{vx}$. By definition
of a sibling $\si_d$-invariant geolamination, there is a leaf $\ol{vy}$
with $\si_d(\ol{vy})=\ol{vx}$. If we now pull back the leaf $\ol{vy}$,
and then keep pulling back this leaf, we will in the end obtain a
branch of the backward orbit of $\ol{vx}$ consisting of countably many
leaves with all these leaves sharing the same endpoint $a$. This
implies that $C$ must be infinite, a contradiction.
\end{proof}

Let us now study infinite cones with periodic vertex.
We write $a_1<a_2<\cdots< a_k$ for points $a_1$, $a_2$, $\dots$, $a_k$
of the unit circle if they appear in the given order under a
counterclockwise (positive) circuit.

\begin{lem}\label{l:inficon1}
Let $\lam$ be a sibling $\si_d$-invariant geolamination. Let $C$ be
an infinite cone of $\lam$ with periodic vertex $v$ of period $n$.
Then all leaves in $C$ are either (pre)critical or (pre)periodic,
and $C$ has the following properties.

\begin{enumerate}

\item There are finitely many leaves $\ol{va_1}, \dots, \ol{va_k}$
    in $C$ such that $v=a_0<a_1<\dots<a_k<v=a_{k+1}$ and
    $\si_d^n(a_i)=a_i$ for each $i$.

\item For each $i$ the set $C'\cap (a_i, a_{i+1})$ is either empty
    or countable.

\item If for some $i$, $C'\cap (a_i, a_{i+1})\ne \0$, then all
    points of $C'\cap (a_i, a_{i+1})$ have $\si_d^n$-preimages in $C'\cap (a_i,
    a_{i+1})$, no preimages elsewhere in $C'$, and $\si_d^n$-images in $[a_i,
    a_{i+1}]\cup\{v\}$.
\end{enumerate}

\end{lem}

\begin{proof}
Let $v$ be $\si_d$-fixed. If $\ell$ is a leaf of $C$, whose forward orbit
consists of infinitely many non-degenerate leaves, then the fact that
$\si_d$ is expanding implies that there will be three distinct
non-degenerate leaves $\ol{va}, \ol{vb}$ and $\ol{vc}$ in $C$ such
that $\si_d$ does not preserve circular orientation on $\{a, b,
c\}$, a contradiction with Lemma~\ref{l:3.17}. This proves the first
part of the lemma and, hence, (2).

Now, (1) is immediate. To prove (3), assume that $C'\cap (a_i,
a_{i+1})\ne\0$ and choose $y\in C'\cap (a_i, a_{i+1})$. By
properties of sibling $\si_d$-invariant geolaminations $\ol{vy}$ has
a preimage $\ol{vx}$ from the same cone. By the choice of points
$a_i$, $x\ne y$. Moreover, by Lemma~\ref{l:3.17} $x\notin (v, a_i)$
(otherwise the circular order is not preserved on $\{v, x, a_i\}$)
and $x\notin (a_{i+1}, v)$ (otherwise the circular order is not
preserved on $\{v, x, a_{i+1}\}$). Hence $y\in (a_i, a_{i+1})$.
Similarly using Lemma~\ref{l:3.17}, we conclude that $\si_d(y)\in
[a_i, a_{i+1}]$ or $\si_d(y)=v$.
\end{proof}

In fact Lemma~\ref{l:3.17} implies a more detailed description of
the dynamics on sets $C'\cap (a_i, a_{i+1})$, which we prove as a
separate lemma.

\begin{lem}\label{l:inficon2}
Let $\lam$ be a sibling $\si_d$-invariant geolamination. Let $C$ be
an infinite cone of $\lam$ with periodic vertex $v$ of period $n$.
Let $\ol{va_1}, \dots, \ol{va_k}$ be all leaves  in $C$ with
$v=a_0<a_1<\dots<a_k<v=a_{k+1}$ and $\si_d^n(a_i)=a_i$ for each $i$.
If, for some $i$, $C'\cap (a_i, a_{i+1})\ne \0$, then there are the
following cases.

\begin{enumerate}

\item The map $\si_d^n$ moves all points of $C'\cap (a_i, a_{i+1})$
    in the positive direction except for those, which are mapped to
    $v$.

\item The map $\si_d^n$ moves all points of $C'\cap (a_i, a_{i+1})$
    in the negative direction except for those, which are mapped to
    $v$.

\item There exist two points $u, w$ with $a_i<u\le w<a_{i+1}$ such
that $\si_d^n(u)=\si_d^n(w)=v$, $C'\cap (u, w)=\0$, the map
$\si_d^n$ maps all points of $(a_i, u]$ in the positive direction
except for those, which are mapped to $v$, and all points of $(w,
a_{i+1}]$ in the negative direction except for those, which are
mapped to $v$.

\end{enumerate}

\end{lem}

\begin{proof}
We may assume that $n=1$. Assume that neither case (1) nor case (2)
holds. Then there are points $x, y\in (a_i, a_{i+1})$ such that
$a_i<x<\si_d(x)=y<a_{i+1}$ and $s, t\in (a_i, a_{i+1})$ such that
$a_i<t=\si_d(s)<s<a_{i+1}$. Take the first pullback $\ol{vx_1}$ of
$\ol{vx}$ in $C$. By Lemma~\ref{l:3.17}, $a_i<x_1<x$. Repeating this
construction, we will find a sequence of leaves $\ol{vx_r}$ of $C$,
which are consecutive pullbacks of $\ol{vx}$ converging to
$\ol{va_i}$ in a ``monotonically decreasing'' fashion. Similarly, we
can find a sequence of leaves $\ol{vs_j}$ of $C$, which are
consecutive pullbacks of $\ol{vs}$ converging to $\ol{va_{i+1}}$ in
a ``monotonically increasing'' fashion.

Applying Lemma~\ref{l:3.17} to pairs of leaves $\ol{vx_r},
\ol{vs_j}$ we see that since for large $r, j$ we have
$a_i<x_r<s_j<a_{i+1}$, then in fact $x<y\le t<s$. Now, take the
greatest (in the sense of the positive order on $[a_i, a_{i+1}]$)
point $x'$ of $C'$, which maps in the positive direction by $\si_d$
to the point $y'=\si_d(x')\in [a_i, a_{i+1}]$ (clearly, $x'$ is
well-defined). Then take the smallest (in the sense of the positive
order on $[a_i, a_{i+1}]$) point $s'$ of $C'$, which maps in the
negative direction by $\si_d$ to the point $t'=\si_d(s')$. By the
above $x'<y'\le t'<s'$. By the choice of $x'$ the $\si_d$-image of
$y'$ is $v$; similarly, $\si_d(t')=v$.
\end{proof}

Observe that Lemmas~\ref{l:finicon} - \ref{l:inficon2} are proven
for all sibling $\si_d$-invariant geolaminations. In the case of
limits of $\si_d$-invariant q-geolaminations we can specify these
results. First we consider finite cones.

\begin{lem}~\label{l:finiconlim}
Let $\lam$ belong to the closure of the set of $\si_d$-invariant
q-geolaminations. Then a finite cone $C$ of $\lam$ with periodic
vertex consists of no more than two leaves.
\end{lem}

\begin{proof}
Suppose otherwise. Then we may find three leaves $\ol{vx}, \ol{vy},
\ol{vz}$ in $C$ with $x<y<z$ each of which is periodic and such that
there are no leaves of $\lam$ separating any two of these leaves in
$\cdisk$. Hence there are two periodic gaps $G$ and $H$, which have
$\ol{vx}, \ol{vy}$ and $\ol{vy}, \ol{vz}$ as their edges,
respectively. By Lemma~\ref{l:limgap3} all these leaves are rigid.
Hence there exists a $\si_d$-invariant q-geolamination $\lam_q$,
which is sufficiently close to $\lam$ and such that $\ol{vx},
\ol{vy}, \ol{vz}$ are leaves of $\lam_q$, a contradiction (clearly,
a $\si_d$-invariant q-geolamination cannot have three leaves with
the same vertex as all its leaves are edges of convex hulls of
equivalence classes).
\end{proof}

Let us now consider infinite cones.

\begin{lem}~\label{l:inficonlim}
Let $\lam$ belong to the closure of the set of $\si_d$-invariant
q-geolaminations. Let $C$ be an infinite cone of $\lam$ with
periodic vertex $v$ of period $n$. Let $\ol{va_1}, \dots, \ol{va_k}$
be all leaves  in $C$ with $v=a_0<a_1<\dots<a_k<v=a_{k+1}$ and
$\si_d^n(a_i)=a_i$ for each $i$. If, for some $i$, $C'\cap (a_i,
a_{i+1})\ne \0$, then there are the following cases.

\begin{enumerate}

\item The map $\si_d^n$ moves all points of $C'\cap (a_i, a_{i+1})$
    in the positive direction except for those, which are mapped to
    $v$.

\item The map $\si_d^n$ moves all points of $C'\cap (a_i, a_{i+1})$
    in the negative direction except for those, which are mapped to
    $v$.

\end{enumerate}

Moreover, for some $0\le r\le k+1$, the map $\si_d$ maps points of
$C'\cap (a_i, a_{i+1})$ in the negative direction for any $0\le i\le
r-1$ and in the positive direction for any $r\le i\le k-1$.
\end{lem}

\begin{proof}
We may assume that $n=1$. First, we claim that case (3) from
Lemma~\ref{l:inficon2} never holds. Indeed, suppose case (3) from
Lemma~\ref{l:inficon2} holds for some $i$. Choose a leaf $\ol{vx}$
of $\lam$ very close to $\ol{va_i}$. There are leaves $\ol{v'x'}$ of
$\si_d$-invariant q-geolaminations very close to $\ol{vx}$ so that
$v'\approx v$ and $x'\approx x$. Since $\si_d$-invariant
q-geolaminations cannot have leaves with periodic endpoints, which
are not periodic, $v'\ne v$. By Lemma~\ref{l:limleaf1}, $\ol{v'x'}$
is disjoint from $\ol{va_i}$. Since $x$ maps in the positive
direction in $(a_i, a_{i+1})$, then $a_i<v'<v$. However, similar
arguments applied to leaves $\ol{vy}$ with $y\in (a_i, a_{i+1}),
y\approx a_{i+1}$ show that $\lam_q$ will have leaves $\ol{v''y''}$
with endpoints $v''\approx v, v<v''<a_i$ and $y''\approx y$.
Clearly, such leaves $\ol{v'x'}$ and $\ol{v''y''}$ will cross, a
contradiction.

The proof of the last claim is similar to the above. Suppose that for
some $i$ the map $\si_d$ moves points of $C'\cap (a_i, a_{i+1})$ in the
positive direction. Then by the previous paragraph all
$\si_d$-invariant q-geolaminations $\lam_q$ have leaves $\ol{v't}$ with
$v'\approx v$ being such that $a_i<v'<v$ and $t\in (a_i, a_{i+1})$
being sufficiently close to $a_i$ so that $a_i<t<\si_d(t)<a_{i+1}$.
This implies that, for any $j>i$ with $C'\cap (a_j, a_{j+1})$, the
leaves $\ol{v''h}$ of $\lam_q$, which are very close to leaves of $C$
connecting $v$ and points in $(a_j, a_{j+1})$ must also have an
endpoint $v''\approx v$ with $a_j<v''<v$ (as otherwise these leaves
would cross leaves $\ol{v't}$ described above). This implies that the
endpoint $h$ of $\ol{v''h}$ is mapped in the positive direction by
$\si_d$ as otherwise $\si_d(\ol{v''h})$ will cross $\ol{v''h}$. Since
leaves $\ol{v''h}$ approximate leaves of $\lam$ this in turn  implies
that points of $C'\cap (a_j, a_{j+1})$ are mapped by $\si_d$ in the
positive direction. This completes the proof.
\end{proof}

\section{The Mandelbrot set as the quotient of the space of
quadratic limit geolaminations}\label{s:mandelquot}

We begin with characterization of limits of q-geolaminations in the
quadratic case. We give an explicit description of geolaminations from
$\olq$. It turns out that each such geolamination $\lam$ can be
described as a specific modification of an appropriate q-geolamination
$\lam^q$ from $\laq_2$. The full statement depends on the kind of
q-geolamination $\lam^q$ involved. For brevity we introduce a few
useful concepts below.

\begin{dfn} By a \emph{generalized critical quadrilateral}
$Q$ we mean either a 4-gon whose $\si_2$-image is a leaf, or a
critical leaf (whose image is a point). A \emph{collapsing \ql} is a
generalized critical \ql{} with four distinct vertices.
\end{dfn}

The notion of generalized critical quadrilateral was used in
\cite{bopt14} where we study cubic (geo)laminations, in particular
those of them, which have generalized critical quadrilaterals as their
critical sets.

\begin{dfn}\label{d:coexist}
Two geolaminations \emph{coexist} if their leaves do not cross.
\end{dfn}

This notion was used in \cite{bopt13}. Observe that, if two
geolaminations coexist, then a leaf of one geolamination is either also
a leaf of the other lamination or is located in a gap of the other
geolamination, and vice versa.

\begin{dfn}\label{d:hyperb}
A $\si_2$-invariant geolamination is called \emph{hyperbolic} if it has
a periodic Fatou gap of degree two.
\end{dfn}

Clearly, if a $\si_2$-invariant geolamination $\lam$ has a periodic
Fatou gap $U$ of period $n$ and of degree greater than one, then the
degree of $G$ is two. By \cite{thu85}, there is a unique edge
$M(\lam)$ of $U$ that also has period $n$. In fact this edge and
its sibling $M'(\lam)$ are the two \emph{majors} of $\lam$ while
$\si_2(M(\lam))=\si_2(M'(\lam))=m(\lam)$ is the \emph{minor} of
$\lam$ \cite{thu85} (recall that a \emph{major} of a
$\si_2$-invariant
geolamination is the longest leaf of $\lam$). 
Any $\si_2$-invariant hyperbolic geolamination $\lam$ is actually a
q-geolamination $\lam_\sim$ corresponding to the appropriate
\emph{hyperbolic $\si_2$-invariant lamination} $\sim$ so that the
topological polynomial $f_\sim$ considered on the entire complex plane
is conjugate to a hyperbolic complex quadratic \emph{hyperbolic}
polynomial; this justifies our terminology.

\begin{dfn}\label{d:criset}
A \emph{critical set} $\cri(\lam)$ of a $\si_2$-invariant geolamination
$\lam$ is either a critical leaf or a gap $G$ such that $\si_2|_G$ has
degree greater than one.
\end{dfn}

A $\si_2$-invariant q-geolamination $\lam$ either has a finite critical
set (a critical leaf, or a finite critical gap) or is hyperbolic.
In both cases, a critical set is unique.
Lemma~\ref{l:critidet} shows that critical sets
are important. This lemma easily follows from results in \cite{thu85};
for completeness we sketch a proof.

\begin{lem}[\cite{thu85}]\label{l:critidet}
Suppose that $\lam$ and $\lam'$ are $\si_2$-invariant geolaminations
such that $\cri(\lam)=\cri(\lam)$ and one of the following holds:

\begin{enumerate}

\item $\cri(\lam)$ has no periodic points;

\item $\cri(\lam)$ has more than two points;

\item $\cri(\lam)=\ol{c}$ is a critical leaf with a periodic
    endpoint and there are two gaps $G$ and $\widehat G$, which
    share $\ol{c}$ as their common edge such that $G$ and $\widehat
    G$ are gaps of both $\lam$ and $\lam'$.
\end{enumerate}

Then $\lam=\lam'$.

\end{lem}

\begin{proof} Consider the collection $\lam^*$ of all leaves obtained by pulling
back all leaves from $\cri(\lam)$ (and in case (3) also from $G\cup
\widehat G$). Any $\si_2$-invariant geolamination satisfying (1),
(2) or (3) must contain $\lam^*$. Hence its closure
$\ol{\lam^*}=\lam''$ is contained in both $\lam $ and $\lam'$.
Moreover, by \cite{thu85} $\lam''$ is $\si_2$-invariant, and by our
construction $\cri(\lam'')=\cri(\lam)$ (and, in case (3) contains
$G$). Clearly, every gap of $\lam''$ (except $\cri(\lam)$ if it is a
gap or $G$ and $\widehat G$ in case (3)) either maps one-to-one to
$\cri(\lam)$ or maps one-to-one to a periodic gap.

Since $\cri(\lam'')=\cri(\lam)$, no leaves of $\lam$ or $\lam'$ can
be contained in $\cri(\lam'')=\cri(\lam)=\cri(\lam')$ or its
preimages (in case (3) no leaves can be contained in $G\cup \widehat
G$ or their $\lam''$ preimages). Since the first return map on the
vertices of a finite periodic gap is transitive \cite{thu85} (i.e.,
all its vertices belong form one periodic orbit under the first
return map), no leaves of $\lam$ or $\lam'$ can be contained in a
finite periodic gap of $\lam''$ or its preimages. Otherwise a
periodic gap $H$ of $\lam''$ may be a Siegel gap with exactly one
(pre)critical edge. In this case the first return map on the
boundary of $H$ is also transitive (similar to the case of a finite
periodic gap) in the sense that \emph{any} point of $H\cap \uc$ has
a dense orbit in $H\cap \uc$ under the first return map. Hence the forward orbit of any
chord inside $H$ contains intersecting chords. Thus, as before
we see that no leaves of $\lam$ or $\lam'$ can be contained in $H$
or its preimages. Finally, a periodic gap $U$ of $\lam''$ can be a
Fatou gap of degree greater than one; however in this case $U$ is
(pre)critical, and this case has been covered before.

We conclude that $\lam=\lam'=\lam''$ as desired.
\end{proof}

For convenience we state Corollary~\ref{c:critidet}.

\begin{cor}\label{c:critidet}
If the critical set $\cri(\lam)$ of a $\si_2$-invariant
geolamination $\lam$ is a gap, then the union of gaps of $\lam$ that
are pullbacks of $\cri(\lam)$ is dense in $\lam$.
\end{cor}

\begin{proof} Consider the collection $\lam^*$ of all edges of gaps
of $\lam$ which are pullbacks of $\cri(\lam)$. Its closure
$\ol{\lam^*}=\lam''$ is contained in $\lam$ and is itself a
$\si_2$-invariant geolamination. By Lemma~\ref{l:critidet}(2), we have
$\lam''=\lam$, as desired.
\end{proof}

If $\cri(\lam)$ is a generalized critical \ql{}, then
$\si_2(\cri(\lam))=m(\lam)$. Lemma~\ref{l:limits} shows the importance
of sibling $\si_2$-invariant geolaminations with generalized critical
\ql s.

\begin{lem}\label{l:limits}
Suppose that a sequence of pairwise distinct $\si_2$-invariant
q-geolaminations $\lam_i$ converges to a $\si_2$-invariant
geolamination $\lam$. Then $m(\lam_i)$ $\to$ $m(\lam)$ and
$\cri(\lam)=\si_2^{-1}(m(\lam))$, which completely determines the
geolamination $\lam$ except when $m(\lam)$ is a periodic point. In
particular, only sibling $\si_2$-invariant geolaminations with critical
sets that are generalized critical \ql s can be limits of non-stabilizing
sequences of $\si_2$-invariant q-geolaminations while $\si_2$-invariant
geolaminations with critical gaps of more than four vertices are
isolated in $\olq$.
\end{lem}

\begin{proof}
By definition of a major of a geolamination and of the convergence in
the Hausdorff metric, majors of $\lam_i$ converge to a major of $\lam$.
Hence $m(\lam_i)\to m(\lam)$. Now, suppose that $\cri(\lam)$ is not a
generalized critical \ql. Then by \cite{thu85} $\cri(\lam)$ has more
than four vertices and $\si_2(\cri(\lam))$ is a preperiodic gap.
However, by Lemma~\ref{l:limgap3}, the set $\cri(\lam)$ is rigid, and
hence geolaminations $\lam_i$ must have $\cri(\lam)$ from some time on
as their critical set. This implies by Lemma~\ref{l:critidet} that all
$\lam_i$'s are equal, a contradiction. If $m(\lam)$ is not a periodic
point, then, by Lemma~\ref{l:critidet}, the fact that
$\cri(\lam)=\si_2^{-1}(m(\lam))$ completely determines $\lam$. The
rest of the lemma follows.
\end{proof}

Theorem~\ref{t:limitlam} describes all geolaminations from $\olq$. A
periodic leaf $\n$ such that the period of its endpoints is $k$ and
all leaves $\n, \si_2(\n), \dots, \si_2^{k-1}(\n)$ are pairwise
disjoint, is said to be a \emph{fixed return} periodic leaf.


\begin{thm}\label{t:limitlam}
A geolamination $\lam$ belongs to $\olq$ if and only there exists a
unique maximal q-geolamination $\lam^q$ coexisting with $\lam$, and
such that either $\lam=\lam^q$ or $\cri(\lam)\subset \cri(\lam^q)$ is a
generalized critical quadrilateral and exactly one of the following
holds.

\begin{enumerate}

\item $\cri(\lam^q)$ is finite and the minor $\si_2(\cri(\lam))=m(\lam)$ of
    $\lam$ is a leaf of $\lam^q$.

\item $\lam^q$ is hyperbolic with a critical Fatou gap $\cri(\lam)$
    of period $n$, and exactly one of the following holds:

\begin{enumerate}

\item $\cri(\lam)=\ol{ab}$ is a critical leaf with a periodic
    endpoint of period $n$, and $\lam$ contains exactly two
    $\si_2^n$-pullbacks of $\ol{ab}$ that intersect $\ol{ab}$
    (one of these pullbacks shares an endpoint $a$ with
    $\ol{ab}$ and the other one shares an endpoint $b$ with
    $\ol{ab}$).

\item $\cri(\lam)$ is a collapsing \ql{} and $m(\lam)$ is a
    fixed return periodic leaf.



\end{enumerate}

\end{enumerate}

Thus, any q-geolamination corresponds to finitely many geolaminations
from $\olq$ and the union of all of their minors is connected.
\end{thm}

\begin{proof}
Let $\lam\in\olq\sm \laq_2$. Then $\lam=\lim \lam_i$, where $\lam_i\in \laq_2$.
By \cite{thu85}, the minor $m(\lam)$ is a leaf $\m$ of $\lam_\qml$ or an
endpoint of such a leaf $\m$, or a point $\m$ of $\uc$ that is a class
of $\qml$. Clearly, the full $\si_2$-preimage $\si_2^{-1}(m(\lam))$ of
$m(\lam)$ is a collapsing quadrilateral or a critical chord.

Suppose that $\m$ has no periodic vertices. Then by \cite{thu85}
there is a unique $\si_2$-invariant lamination $\sim$ such that
$T=\cri(\lam_\sim)$ is the convex hull of a $\sim$-class and either
(1) $T$ is a leaf and $\m=\si_2(T)$, or (2) $T$ is a \ql{} and
$\m=\si_2(T)$, or (3) $T$ has more than four vertices, $\m$ is an
edge of $\si_2(T)$ which is a preperiodic gap all of whose edges
eventually map to leaves from the same cycle of leaves. Moreover, by
\cite{thu85} the set $\si_2(T)$ is the convex hull of a class of
$\qml$. Finally, by \cite{thu85} for each pair of sibling
edges/vertices $N, N'$ of $T$ we can add their convex hull $\ch(N,
N')$ to $\lam_\sim$ (i.e., insert two leaves connecting appropriate
endpoints of $N, N'$) and then all appropriate eventual pullbacks of
$\ch(N, N')$ to $\lam_\sim$ to create a $\si_2$-invariant
geolamination with $\si_2(N)$ as its minor.

Observe that by Lemma~\ref{l:critidet} the geolamination with critical
set $\ch(N, N')$ is unique. Thus, all edges and vertices of $\si_2(T)$
are minors of these $\si_2$-invariant geolaminations with collapsing
\ql s; moreover, the minor $m(\lam_\sim)$ is also an edge of $\si_2(T)$
and serves as the minor of two $\si_2$-invariant geolaminations (one of
them is $\lam_\sim$, the other one has the critical set $\ch(M, M')$
where $M, M'$ are majors of $\lam_\sim$). On the other hand, by
\cite{thu85} the $\si_2$-invariant geolaminations other than the just
described have minors disjoint from $\si_2(\cri(\lam_\sim))$. Thus, our
originally given geolamination $\lam$ is one of the just described
geolaminations. Observe that $m(\lam)$ is a vertex or an edge of
$\si_2(T)$ not coinciding with $\si_2(T)$ (if these two sets coincide,
then $\cri(\lam)=T$ and by Lemma~\ref{l:critidet}, we have $\lam=\lam_\sim$, a
contradiction).

Let us show that in fact any $\si_2$-invariant geolamination $\hlam$
with critical set $\ch(N, N')$ (here $N, N'$ are sibling edges of $T$)
is the limit of a sequence of pairwise distinct $\si_2$-invariant
geolaminations. Indeed, by \cite{thu85}, each edge of $T$ can be
approached by non-periodic leaves of $\lam_\qml$ that are convex hulls of classes
of $\qml$, and each vertex of $T$ can be approached by a degenerate
non-periodic class of $\qml$. Choose the $\si_2$-invariant q-geolaminations for
which these leaves/points are minors. We may assume that they converge
to a limit geolamination $\lam'$. This implies that $\lam'$ has a
collapsing \ql{} or a critical leaf as the critical set (the limit of
collapsing \ql s/critical leaves is a collapsing \ql/critical leaf);
clearly, this limit collapsing \ql/critical leaf must coincide with
$\ch(N, N')$. By Lemma~\ref{l:critidet}, this implies that
$\lam'=\hlam$ as claimed.

Assume now that $\m$ has a periodic vertex of period $n$. Then, by
\cite{thu85}, there is a $\si_2$-invariant lamination $\sim$ with the
following properties.
The minor of $\lam_\sim$ is $\m$, and a major $M$ of $\lam_\sim$ is an
edge of a critical Fatou gap $U=\cri(\lam_\sim)$ of period $n$.
Recall that the minor $m(\lam)$ of
$\lam$ is either $\m$ itself, or an endpoint of $\m$. By
Lemma~\ref{l:limits}, we have $m(\lam_i)\to m(\lam)$,
$\cri(\lam)=\si_2^{-1}(m(\lam))$, and if $m(\lam)=\m$ is non-degenerate,
then, by Lemma~\ref{l:limits}, the geolamination $\lam$ is completely determined. Moreover,
assume that $M$ is an edge of a finite periodic gap $G$ of $\lam_\sim$
(informally speaking, $U$ ``rotates'' about $G$ under the appropriate power of $\si_2$). If
$m(\lam)=\m$, then $\lam$ has two finite gaps $G$ and $\si_2^{-1}(\m)$
sharing a periodic edge $M$, a contradiction with
Lemma~\ref{l:limgap1.5}. Moreover, assume that $M$ is ``flipped'' by
the appropriate power, say, $\si_2^k$ of $\si_2$ (then $U$ maps by
$\si^k_2$ to $\si_2^k(U)$, which shares $M$ with $U$ as their common
edge). In this case $\lam$ has gaps $\si_2^{-1}(\m)$ and $\si_2^k(U)$
that share a periodic edge $M$, again a contradiction with
Lemma~\ref{l:limgap1.5}.

To sum it all up, if $m(\lam)=\m$ and hence by Lemma~\ref{l:limits}
we have $\cri(\lam)=\si_2^{-1}(\m)$, then $\m$ must have pairwise
disjoint images until it maps back to itself by $\si_2^n$. By
Lemma~\ref{l:limits}, the geolamination $\lam$ is completely
determined by the fact that $\cri(\lam)=\si_2^{-1}(\m)$ (actually,
$\lam$ can be constructed by pulling back the \ql{} $\si_2^{-1}(\m)$
in a fashion consistent with $\lam_\sim$). On the other hand, if
$\m$ has pairwise disjoint images until it maps back to itself by
$\si_2^n$, then, by \cite{thu85}, the minor $\m$ can be approached by pairwise
disjoint minors of $\si_2$-invariant q-geolaminations. Assuming that
these geolaminations converge, we see that the thus created limit
geolamination must coincide with the above described geolamination
$\lam$. This covers case (2-b).

To consider case (2-a), assume that $m(\lam)$ is an endpoint of $\m$.
By Lemma~\ref{l:limits}, we have $\cri(\lam)=\si_2^{-1}(m(\lam))$, and by the
assumption we may set $\cri(\lam)=\ol{ab}$ where $a$ is an endpoint of
$M$. Properties of $\lam_\sim$ imply that there exist unique pullbacks
of $\ol{ab}$ under the maps $\si_2,$ $\si_2^2,$ $\dots,$ $\si_2^{n-1}$
with endpoints $\si_2^{n-1}(a),$ $\si_2^{n-2}(a),$ $\dots,$ $\si_2(a)$.
However there are two possible pullbacks of $\ol{ab}$ coming out of
$a$. It is easy to see that these two chords are contained in $U$ on
distinct sides of $\ol{ab}$. However they both cannot be leaves of
$\lam$. Indeed, all leaves of $\lam$ with endpoint $a$ form a cone, and
if both chords are leaves of $\lam$, we will have a contradiction with
Lemma~\ref{l:inficonlim}. Consider both cases separately and show that
the corresponding geolamination $\lam$ is completely determined by the
choice of the $\si_2^n$-pullback leaf of $\ol{ab}$ with endpoint $a$.
For definiteness, assume that $M=\ol{xa}$ where $(x, a)$ is a hole of
$U$ (i.e., $(x, a)\cap U'=\0$) and let $M'=\ol{x'b}$ be the edge of $U$,
which is the sibling leaf of $M$.

(1) Assume that $\ol{ad}, a<d<b$ is the $\si_2^n$-pullback of $\ol{ab}$
with endpoint $a$ contained in $U$. Then the sibling $\ol{bd'}$ of
$\ol{ad}$ is also a leaf of $\lam$. Pulling it back under $\si_2^n$, we
see that there is a concatenation $L$ of $\si_2^n$-pullbacks of
$\ol{bd'}$, which accumulates to $x$. Moreover, suppose that $\lam$ has
a leaf $\ol{by}$ where $x<y<a, y\approx a$. Then it follows that
$\si_2^n(\ol{by})$ crosses $\ol{by}$, a contradiction. Hence such
chords are not leaves of $\lam$, which implies that $\lam$ has a gap $G$
with vertices $a, b$ and other vertices belonging to $(b, a)$. By
properties of $\si_2$-invariant geolaminations there exists the
``sibling gap'' $\widehat G$ of $G$ located on the other side of
$\ol{ab}$ ($\widehat G$ is actually a rotation of $G$ by half of the
full rotation). Moreover, the existence of $L$ implies that $G$ has an
edge $\ol{bz}$ and one of the following holds: either (a) $z=d'$, or
(b) $x\le z<a$. It turns out that either of these two cases is realized
depending on the dynamics of $M$; moreover, we will show that to each
of the cases corresponds a unique geolamination, which is completely
determined by the choices we make.

(a) Assume that $M$ is a fixed return periodic leaf. Let us show
that then $z=x$. Indeed, by Corollary~\ref{c:critidet}, pullbacks of
$U$ are dense in $\lam_\sim$. In particular, there are pullbacks of
$U$ approximating $M$ from the outside of $U$. Each pullback of
$U$ which is close to $M$ has two long edges, say, $N$ and $R$,
which converge to $M$ as pullbacks in question converge to $M$.
Choose the pullback $V$ of $U$ so close to $M$ that $N$ and $R$ are
longer than the two edges of $\ch(M, M')$ distinct from $M, M'$. Let
us show that $N$ and $R$ are themselves pullbacks of $M$ and $M'$.
Indeed, let $N'$ and $R'$ be the sibling leaves of $N$ and $R$. By
the Central Strip Lemma (Lemma II.5.1 of \cite{thu85}) and by the
choice of the pullback of $U$ very close to $M$ we see that when,
say, $N$ enters the strip between $N$ and $N'$ it can only grow in
size.

Repeating this argument, we will finally arrive at the moment when
$V$ maps to $U$ when by the above $N$ and $R$ will map to $M$ and
$M'$. This implies that there exists a chord connecting the
appropriate endpoints of $N$ and $R$ and mapping to $\ol{ab}$ at the
same moment. In other words, this shows that there are chords very
close to $M$ which map onto $\ol{ab}$ under a certain iteration of
$\si_2$ and are disjoint from $U$ before that. Hence these chords
must be leaves of $\lam$. Since by construction they accumulate upon
$M$, we see that $M$ must be a leaf of $\lam$. Therefore $d'\le z\le
x$.

Now, assume that $z=d'$. Then the properties of geolaminations
easily imply that $L$ is a part of the boundary of $G$ and that
$\si_2^n(G)=G$. We claim that the only $\si_2^n$-critical edge of
$G$ is $\ol{ab}$. Indeed, no edge of $G$ from $L$ is
$\si_2^n$-critical. On the other hand, if $\ell\subset K$ is a
$\si_2^n$-critical edge of $G$, then a forward image $\si_2^i(\ell),
0<i<n$ of $\ell$ must coincide with $\ol{ab}$. However, this would
imply that either $\si_2^i(G)=G$ (contradicting the fact that $i<n$
and the period of $G$ is $n$) or $\si_2^i(G)=\widehat G$
(contradicting the fact that $\widehat G$ is not periodic). By
Lemma~\ref{l:limgap1.5}(1) edges of $G$ cannot be periodic. Consider
the rest of the boundary of $G$ whose vertices belong to $[x, a]$;
denote this subarc of $\bd(G)$ by $K$. Since $x$ and $a$ are
$\si_2^n$-fixed and the degree of $\si_2^n$ on $\bd(G)$ is one, then
$\si_2^n(K)=K$. Hence it follows that $K$ contains neither
(pre)periodic nor (pre)critical edges of $G$, a contradiction to
Lemma~\ref{l:edges}.

Assume next that $z\ne d'$ is a vertex of $L$. Then
$\si_2(\ol{zb})=\ol{\si_2(z)a}$ which crosses $\ol{bz}$, a
contradiction. This leaves the only possibility for $z$, namely that
$z=x$ and so $G=\ch(a, b, x)$ and $\widehat G=\ch(a, b, x')$. Since
$\ol{ab}$ is isolated in $\lam$, we can remove it from $\lam$ and
thus obtain a new geolamination $\lam'$ which, as follows from
Lemma~\ref{l:critidet}, is completely determined by the critical
\ql{} $\ch(a, x, b, x')$. Adding $\ol{ab}$ and all its pullbacks to
$\lam'$ we finally see that $\lam$ is completely determined by the
fact that $\ol{ad}$ is a leaf of $\lam$ and the fact that $M$ is a
fixed return periodic leaf. In fact, $\lam$ can be viewed as the
$\si_2$-invariant geolamination determined by the choice of the
collapsing \ql{} $\ch(M, M')$ and then inserting in it the critical
leaf $\ol{ab}$.

(b) If $M$ is not a fixed return periodic leaf, there are two
subcases: (i) the orbit of $M$ is the union of edges of several
finite gaps permuted by the corresponding power of $\si_2$, and (ii)
$n=2k$ and $M$ is ``flipped'' by $\si_2^k$. Since the arguments are
very similar, we only consider the case (i). Assume that $n=kl$ and
that the orbit of $M$ consists of edges from the boundaries of $k$
pairwise disjoint $l$-gons $D_1, \dots, D_k$, cyclically permuted
under $\si_2$.

We claim that $M$ is not a leaf of $\lam$. Suppose otherwise. Then
there are two gaps of $\lam$ which share $M$ as their edge. On the
one side of $M$ it is a finite $l$-gon, say, $D_1$ with edge $M$,
one of the above mentioned $l$-gons. On the other side of $M$, it is
the gap $G$ constructed above. Since $G$ cannot be a periodic Fatou
of degree greater than one (it is either a collapsing triangle
$\ch(a, b, x)$ or an infinite gap with concatenation $L$ on its
boundary), we get a contradiction with Lemma~\ref{l:limgap1.5}.
Thus, $M$ and no leaf from its orbit is a leaf of $\lam$.

By pulling $L$ back an appropriate number of times, we obtain a
gap $G$, whose boundary consists of $l$ pullbacks of $L$ concatenated
to each other at vertices of $D_1$; observe again that the edges of
$D_1$ are \emph{not} leaves of $\lam$. This also defines the gap
$\widehat G$. Observe that the existence of these two gaps by
Lemma~\ref{l:critidet} completely determines the corresponding
geolamination as pullbacks of all leaves are now well-defined.

(2) Assume that $\ol{da}$, where $b<d<a$, is the $\si_2^n$-pullback of
$\ol{ab}$ with endpoint $a$ contained in $U$. Then the sibling
$\ol{d'b}$ of $\ol{ad}$ is also a leaf of $\lam$. Pulling $\ol{d'b}$
back under $\si_2^n$, we see (similarly to (1) above) that there is
a concatenation $L'$ of $\si_2^n$-pullbacks of $\ol{d'b}$ with
endpoints contained in $(a, x')$; clearly, $L'$ converges to $a$.
Clearly, $L'$ together with the leaves $\ol{d'b}$ and $\ol{ba}$ form
a Jordan curve. It is easy to verify that any chord connecting two
non-adjacent vertices of $L'$ will cross itself under $\si_2^n$. On
the other hand, the leaf $\ol{ad'}$ cannot exist by
Lemma~\ref{l:inficonlim} (recall that $\ol{ad}\in \lam$). Thus, this
Jordan curve is in fact the boundary of a gap $G$ of $\lam$. The
centrally symmetric to it ``sibling gap'' $\widehat G$ together with
$G$ forms a pair of gaps which $\lam$ must have. By
Lemma~\ref{l:critidet} this completely determines the geolamination
$\lam$. Note that in this case pullbacks of $\ol{ad}$ converge to
$M$ and hence $M$ belongs to $\lam$.

Clearly, the same arguments would apply if $m(\lam)$ were the other
endpoint of $\m$. This completes the description of possible limits of
$\si_2$-invariant q-geolaminations with minors contained in a periodic
minor $\m$ from $\lam_\sim$. We see that to each pair of possible
$\si_2^n$-pullbacks of $\ol{ab}$ there is a unique geolamination which
potentially can be the limit of a sequence of $\si_2$-invariant
q-geolaminations. To show that all the described geolaminations are
indeed limits of sequences of $\si_2$-invariant q-geolaminations, we
need to show that each pair of defining pullbacks of $\ol{ab}$ is
possible and that the geolamination described in (2-b) is also the
limit of a sequence of $\si_2$-invariant q-geolaminations.

To prove the latter, note that by \cite{thu85} we can approximate
the fixed return periodic major $M$ of $\lam$ by majors $M_i$ of
$\si_2$-invariant q-geolaminations $\lam_i\to \lam$ ($M_i$'s are
outside the critical Fatou gap of $\lam$). Then by
Lemma~\ref{l:limits} the collapsing \ql{} $\ch(M, M')$ is the
critical set of $\lam$ and $\lam$ is uniquely determined by that.
This completes (2-b).

Consider now (2-a). Assume that $\m$ is a periodic minor, $M=\ol{xa}$
is the corresponding periodic major, $\ol{ab}$ is the critical leaf,
$x<a<b$, the points $a, x$ are of period $n$, and we want to show that
the $\si_2$-invariant geolamination $\lam$ with $\si_2^n$-pullbacks of
$\ol{ab}$ being $\ol{ad}, \ol{d'b}$ with $b<d'<a<d$ described above is
the limit of a sequence of $\si_2$-invariant q-geolaminations. To show
that, consider a critical leaf $\ell=\ol{a_1b_1}$ with $a_1<a<b_1<b$
very close to $\ol{ab}$. Then the fact that $a$ is repelling for
$\si_2^n$ shows that the appropriate $\si_2^n$-pullbacks of $\ol{ab}$
are indeed close to $\ol{ad}$ and $\ol{d'b}$, and converge to $\ol{ad}$
and $\ol{d'b}$ as $\ell\to \ol{ab}$. Moreover, $\ell$ can always be
chosen to correspond to a $\si_2$-invariant q-geolamination of which
$\ell$ will be the critical leaf (the convex hull of the critical class
of the corresponding lamination). Thus, these particular pullbacks can
be realized on a limit geolamination. Equally simple arguments show
that in fact all possibilities listed in the theorem can be realized.
\end{proof}

To interpret the Mandelbrot set as a specific quotient of the closure
$\olq$ of the family $\laq_2$ of all $\si_2$-invariant
q-geolaminations, we define a special equivalence relation on $\olq$.
The definition itself is based upon the fact that by
Lemma~\ref{l:limits} any geolamination from $\olq$ either belongs to
$\laq_2$ or has a critical leaf, or has a critical quadrilateral
($\si_d$-invariant geolaminations with similar properties are called
\emph{quadratically critical geolaminations}, or simply
\emph{qc-geolaminations} \cite{bopt14}).

\begin{dfn}\label{d:minor}
Let $\lam', \lam''\in \olq$. Then the geolaminations $\lam_1$ and
$\lam_2$ are said to be \emph{minor equivalent} if there exists a
finite collection of geolaminations $\lam_1=\lam',$ $\lam_2,$
$\dots,$ $\lam_k=\lam''$ from $\olq$ such that for each $i, 1\le
i\le k-1$ the minors $m(\lam_i)$ and $m(\lam_{i+1})$ of the
geolaminations $\lam_i$ and $\lam_{i+1}$ are non-disjoint.
\end{dfn}

Theorem~\ref{t:limitlam} allows one to explicitly describe classes of
minor equivalence. Namely, by \cite{thu85} and Theorem~\ref{t:limitlam}
to each class $\g$ of $\qml$ we can associate the corresponding
$\si_2$-invariant q-geolamination $\lam_\g$ and finitely many limit
geolaminations $\lam$ of non-constant sequences of $\si_2$-invariant
q-geolaminations $\lam_i$ such that the minor $m(\lam)$ is the limit of
minors $m(\lam_i)$ of $\lam_i$ and is non-disjoint from (actually,
contained in) $\ch(\g)$. Let $\psi:\olq\to \uc/\qml$ be the map which
associates to each geolamination $\lam\in \olq$ the $\qml$-class of the
endpoints of the minor $m(\lam)$ of $\lam$. By Lemma~\ref{l:limits}, we
obtain the following theorem.

\begin{thm}\label{t:main}
The map $\psi:\olq\to \uc/\qml$ is continuous. Thus, the partition of
$\olq$ into classes of minor equivalence is upper semi-continuous and
the quotient space of{} $\olq$ with respect to the minor equivalence is
homeomorphic to $\uc/\qml$.
\end{thm}

\bibliographystyle{amsalpha}

\begin{thebibliography}{9999}

\bibitem[AK79]{ak79} J. Auslander, Y. Katznelson, \emph{Continuous
maps of the circle without periodic points}, Israel J. Math.
\textbf{32} (1979), no. 4, 375–-381.

\bibitem[Blo84]{blo84} A. M. Blokh, \emph{On transitive mappings of
one-dimensional branched manifolds. (Russian)},
Differential-difference equations and problems of mathematical
physics (Russian), \textbf{131} (1984), 3-–9, Akad. Nauk Ukrain.
SSR, Inst. Mat., Kiev.

\bibitem[Blo86]{blo86} A. M. Blokh, \emph{
Dynamical systems on one-dimensional branched manifolds. I.
(Russian)} Teor. Funktsii Funktsional. Anal. i Prilozhen.
\textbf{46} (1986), 8--18; translation in J. Soviet Math.
\textbf{48} (1990), no. 5, 500–-508.

\bibitem[Blo87a]{blo87a} A. Blokh, \emph{Dynamical systems on
one-dimensional branched manifolds. II. (Russian)}, Teor. Funktsii
Funktsional. Anal. i Prilozhen. \emph{47} (1987), 67--77;
translation in J. Soviet Math. \textbf{48} (1990), no. 6, 668–-674.

\bibitem[Blo87b]{blo87b} A. Blokh, \emph{Dynamical systems on
one-dimensional branched manifolds. III. (Russian)} Teor. Funktsii
Funktsional. Anal. i Prilozhen, \textbf{48} (1987), 32--46;
translation in J. Soviet Math. \textbf{49} (1990), no. 2, 875–-883




\bibitem[BL02]{bl02} A. Blokh, G. Levin, \emph{Growing trees, laminations and the dynamics on
the Julia set}, Ergodic Theory and Dynamical Systems \textbf{22}
(2002), 63--97.

\bibitem[BMOV13]{bmov13} A. Blokh, D. Mimbs, L. Oversteegen, K. Valkenburg,
\emph{Laminations in the language of leaves}, Trans. Amer. Math.
Soc., \textbf{365} (2013), 5367--5391.



\bibitem[BOPT13]{bopt13} A. Blokh, L. Oversteegen, R. Ptacek, V. Timorin,
\emph{Laminations from the Main Cubioid}, preprint arXiv:1305.5788
(2013)

\bibitem[BOPT14]{bopt14} A. Blokh, L. Oversteegen, R. Ptacek, V. Timorin,
\emph{Combinatorial Models for Spaces of Cubic Polynomials},
preprint arXiv:1405.4287 (2014)




\bibitem[DH85]{hubbdoua85} A. Douady, J. H. Hubbard, \emph{\'Etude
    dynamique des polyn\^omes complexes} I, II, Publications
    Math\'ematiques d'Orsay \textbf{84-02} (1984), \textbf{85-04}
    (1985).

\bibitem[Kel00]{kel00} K. Keller, \emph{Invariant factors, Julia equivalences and the
(abstract) Mandelbrot set}, Lecture Notes in Mathematics,
\textbf{1732}, Springer-Verlag, Berlin, 2000.


\bibitem[Kiw02]{kiw02} J. Kiwi, \emph{Wandering orbit portraits},
Trans. Amer. Math. Soc. \textbf{354} (2002), 1473--1485.



\bibitem[Mil00]{mil00} J. Milnor, \emph{Dynamics in One Complex Variable},
Annals of Mathematical Studies \textbf{160}, Princeton (2006).


\bibitem[Sch09]{sch09} D. Schleicher, \emph{Appendix: Laminations,
Julia sets, and the Mandelbrot set}, in: ``Complex dynamics:
Families and Friends'', ed. by D. Schleicher, A K Peters (2009),
111--130.

\bibitem[Thu85]{thu85} W.~Thurston. \newblock {\em The
    combinatorics of iterated rational maps} (1985), with appendix by D. Schleicher,
    \emph{Laminations, Julia sets, and the Mandelbrot set}, published
    in: ``Complex dynamics: Families and Friends'', ed. by D.
    Schleicher, A K Peters (2009), 1--137.

\end{thebibliography}

\end{document}